\documentclass[11pt,a4]{amsart}

\usepackage{amsmath}
\usepackage{cancel}
\usepackage{graphicx}
\usepackage{url,hyperref}
\usepackage[style=american]{csquotes}
\usepackage{enumerate}
\usepackage[width=\textwidth]{caption}
\usepackage{subcaption}
\DeclareCaptionSubType*[arabic]{figure}
\usepackage[normalem]{ulem}

\usepackage{own-macros}

\begin{document}

%\title[$(r,D,R)$-Blaschke-Santal\'o diagram]{On $(r,D,R)$-Blaschke-Santal\'o diagrams in Minkowski spaces and beyond: regular $k$-gons}
\title[$(r,D,R)$-Blaschke-Santal\'o diagram]{On $(r,D,R)$-Blaschke-Santal\'o diagrams with regular $k$-gon gauges.}

\author{Ren\'e Brandenberg}
\address{Technical University of Munich, 85747 Garching bei M\"unchen, Germany} \email{brandenb@ma.tum.de}

\author{Bernardo Gonz\'alez Merino}
\address{Departamento de Did\'actica de las Ciencias Matem\'aticas y Sociales, Universidad de Murcia, 30100, Murcia, Spain}
\email{bgmerino@um.es}

%\address{Departamento de Matem\'aticas, Universidad de Murcia, Campus
%Espinar\-do, 30100-Murcia, Spain}

\thanks{This research is a result of the activity developed within the framework of the Programme in Support
of Excellence Groups of the Regi\'on de Murcia, Spain, by Fundaci\'on S\'eneca, Science and Technology
Agency of the Regi\'on de Murcia. The second author is partially supported by Fundaci\'on S\'eneca project 19901/GERM/15,
Spain, and by MICINN Project PGC2018-094215-B-I00 Spain.}

\subjclass[2020]{}


\begin{abstract}

%\BSays{v2}
%We study Blaschke-Santal\'o diagrams for the inradius, diameter, and circumradius measured with respect to a gauge body. We show some natural structural properties, for instance, the fact that the union of all those diagrams is given by the diagram measured with respect to a triangle, or the intersection of all of them coincides with the diagram measured respect to a parallelotope, and only that. When restricting to symmetric gauges, the union of diagrams is unexpectedly not given by the diagram of a regular hexagon, neither a single diagram. Besides these fundamental properties given by the diagrams of $k$-gons, for $k=3,4,6$, we also show that the diagram of $5$-gons provides us of a useful example to understand Jung's inequality.

%\BSays{v3}
%In this paper we develop an insightful 
We provide a study of Blaschke-Santal\'o diagrams for the inradius, diameter, and circumradius, measured with respect to different gauges. This contrasts previous works on those diagrams, which are all considered for euclidean measure. By proving several new inequalities and properties between these three functionals, we compute the intersection and the union over all possible gauges of those diagrams, showing that they coincide with the corresponding diagrams of a parallelotope and (in the planar case) a triangle, respectively. Further new extremal properties are derived considering the diagrams with respect to a regular pentagon or hexagon gauge.
\end{abstract}

\maketitle

\section{Introduction}

Let $\mathcal K^n$ be the \emph{set of convex bodies} (i.e. convex and compact sets) \emph{in $\R^n$} and $\mathcal K^n_0$ be the subset of $\mathcal K^n$ formed by \emph{centrally symmetric sets}. % , i.e. those with $-K=c+K$ for some $c\in\R^n$.
For any $K,C\in\mathcal K^n$, let $R(K,C)$ be the \emph{circumradius} of $K$ with
respect to $C$, i.e.~the smallest value $\lambda\geq 0$ such that a translate of $K$
is contained in $\lambda C$, and $r(K,C)$ be the \emph{inradius} of $K$ with respect to $C$, i.e.~the largest value $\lambda\geq 0$ such that a translate of $K$ contains $\lambda C$. The second set $C$ is usually fulldimensional and therefore called the \emph{gauge} the functionals are based on.
Finally, let $D(K,C)$ be the \emph{diameter} of $K$ with respect to $C$, i.e.~twice the maximal circumradius $R(\{x,y\},C)$ for some $x,y\in K$.\footnote{There are other generalizations of the diameter for general gauges, but this one is the most common.}

The aim of this paper is to describe the range of values that the inradius, circumradius and diameter of $K$ and $C$ may achieve, for some fixed gauges $C$, but also for varying $K$ and $C$.
%\BSays{do we really need to make this distinction here? - I mean it is what we do, but maybe it is overexplaining} \RSays{If you don't like it, we can remove. It is ment as a motivation, not as an explanation. The title sounds a bit like \enquote{weird special situations} and I thought we should tell people asap that we do more}. 
To do so, consider the mapping
\[
f:\mathcal K^n\times\mathcal K^n\rightarrow[0,1]^2,\quad f(K,C)=\left(\frac{r(K,C)}{R(K,C)},\frac{D(K,C)}{2R(K,C)}\right).
%=(x(K,C),y(K,C)). 
\]
The set $f(\mathcal K^n,C)$ is the well-known Blaschke-Santal\'o diagram for the inradius, circumradius and diameter with respect to $C$, or $(r,D,R)$-diagram for short. This naming honours two important mathematicians. Blaschke on the one side, who considered in 1916 the corresponding diagram for the volume, surface area and mean width of $3$-dimensional convex bodies \cite{Bl16}.  Santal\'o \cite{Sa61} on the other side, who described in 1961 several such diagrams, %for instance,
%the one of the inradius, circumradius and diameter in the Euclidean plane, i.e. measuring the functionals with respect to the
%\emph{Euclidean norm} $\|x\|_2$, $x\in\mathbb R^n$, whose \emph{(Euclidean) unit ball} is denoted by $\mathbb B_2=\{x\in\R^n:\|x\|_2\leq 1\}$. 
involving three functionals out of area, perimeter, circumradius, inradius, diameter, and minimum width of planar convex sets, and who completely described $f(\CK^2,\B_2)$, where $\B_2 = \{x\in\R^n : \|x\|_2 \leq 1\}$ denotes the \emph{euclidean unit ball}. To do so Santal\'o proved the validity of the inequality
\[
\begin{split}
2R(K,\mathbb B_2)(2R(K,\mathbb B_2)+\sqrt{4R(K,\mathbb B_2)^2-D(K,\mathbb B_2)^2})r(K,\mathbb B_2) & \\
\geq D(K,\mathbb B_2)^2\sqrt{4R(K,\mathbb B_2)^2-D(K,\mathbb B_2)^2} &
\end{split}
\]
for all $K\in\mathcal K^2$ and observed that, together with the well-known inequalities
\[
\begin{split}
    D(K,\B_2) & \leq 2R(K,\B_2) \\
    r(K,\B_2)+R(K,\B_2) & \leq D(K,\B_2) \\
    %\sqrt{\frac{n}{2(n+1)}} D(K,\B_2) & \leq R(K,\B_2)
     \sqrt{2(n+1)}R(K,\B_2) & \leq \sqrt{n} D(K,\B_2)
\end{split}
\]
fully described the diagram $f(\mathcal K^2,\B_2)$ (see Figure \ref{fig:EuclideanDiagram}).

\begin{figure}
    \centering
    \includegraphics[width=8cm]{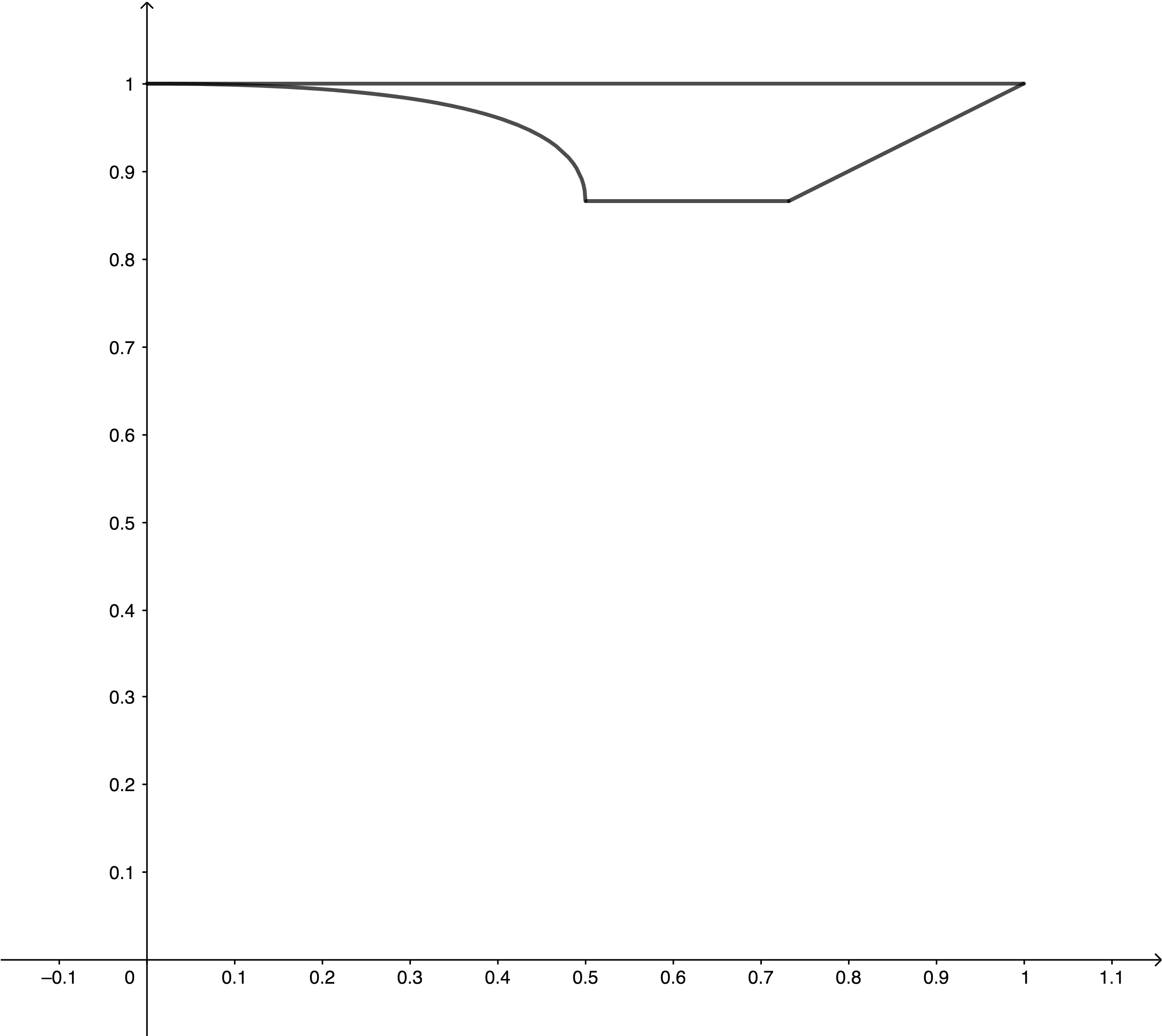}
    \caption{The (r,D,R)-diagram $f(\mathcal K^2,\B_2)$.}
    \label{fig:EuclideanDiagram}
\end{figure}

When investigating such diagrams one typically discovers that lots of special classes of convex sets describe the boundaries and that they do helps us to understand even more about their speciality.
%The diagram $f(\mathcal K^2,C)$ plays a major role in understanding properties of very special families of convex bodies: 
Centrally symmetric, parallelotopes, constant width set, complete or reduced sets, or simplices are just examples of bodies filling different boundaries of those diagrams.

After Santal\'o's paper, several authors gave full descriptions of $2$-dimensional diagrams for planar sets \cite{BoHCSa03, BuBuFi, BuPr20, DeHePr, Ft20, FtHe, FtLa21, HC00, HC02, HCSG00, HCT, LuZu21}, see also \cite{AnFr, AnHe, BeBuPr}, for higher dimensional
sets \cite{HCPSSG03, HCSaGo}, or even $3$-dimensional diagrams (see \cite{BrGM17} and the incomplete description in \cite{TiKe}). But all these diagrams have been considered for euclidean spaces only. 

However, many functionals can be naturally extended to Minkowski (or Banach) spaces, or even further to generalized Minkowski spaces, as we have done above in the case of the circumradius, inradius and diameter. %with respect to some arbitrary $C\in\mathcal K^n$.

In our view, the most significant
%The main
result of this paper is twofold, and we split it into two theorems, so that two aspects can be better understood. The first one fully describes the Blaschke-Santal\'o diagram $f(\CK^2,\CK^2)$.

\begin{thm}\label{thm:dominating_diagram}
Let $K,C \in \CK^2$. Then
\[
\begin{split}
D(K,C) & \leq 2R(K,C)\\
%2r(K,T)+R(K,T) & \leq \frac{3}{2}D(K,T)\\
4r(K,C) + 2R(K,C) & \leq 3 D(K,C)\\
D(K,C)\left(2R(K,C)-D(K,C)\right) & \leq 4r(K,C)D(K,C).
\end{split}
\]
Moreover, those three inequalities completely describe the boundary of $f(\CK^2,\CK^2)$.
\end{thm}

\begin{figure}
    \centering
    \includegraphics[width=8cm]{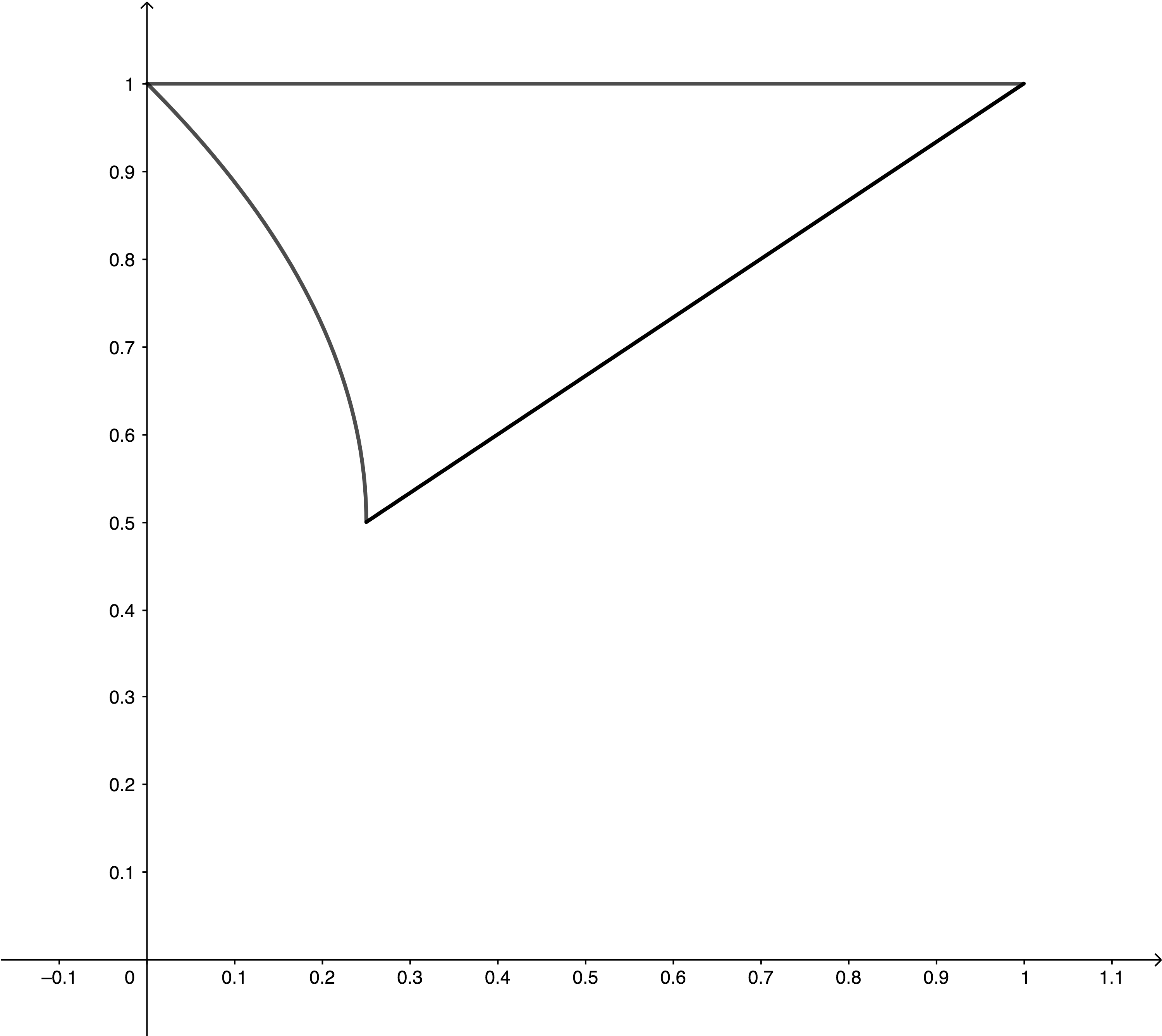}
    \caption{The (r,D,R)-diagram $f(\CK^2,\CK^2)$.}
    \label{fig:TriangleDiagram}
\end{figure}

Notice that the first inequality is basic and well-known, while the third is completely new. The second is the 2-dimensional version of a new inequality that holds true for arbitrary dimensions, which is almost a direct consequence from a result in \cite{BrGM17_2}.

\begin{thm} \label{cor:nr+R<=n+1D/2}
Let $K,C\in\mathcal K^n$. Then
\begin{equation} \label{eq:nr+R<=n+1D/2}
nr(K,C)+R(K,C)\leq\frac{n+1}{2}D(K,C).
\end{equation}
\end{thm}

%Many geometric behaviors relating two or more functionals become extreme in the case that the involved set is a simplex.
%Basic algebraic computations show 
%It is quite obvious that Santal\'o's diagram $f(\mathcal K^2,\B_2)$ is a subset of the diagram $f(\mathcal K^2,S)$ (cf. Figures \ref{fig:EuclideanDiagram} and \ref{fig:TriangleDiagram}).
%This is not a coincidence, which is evident from the following general result. 
%and explains about the structure of the $(r,D,R)$-Blaschke-Santal\'o diagrams in generalized Minkowski spaces, showing that $f(\mathcal K^2,C)$ is contained in $f(\mathcal K^2,S)$,
%for every $C\in\mathcal K^2$ and $S$ being any triangle.

%\begin{thm}\label{thm:dominating_diagram}
%Let $K,C\in\mathcal K^2$. Then
%\[
%\begin{split}
%D(K,C) & \leq 2R(K,C)\\
%4r(K,C) + 2R(K,C) & \leq 3 D(K,C)\\
%D(K,C)\left(2R(K,C)-D(K,C)\right) & \leq 4r(K,C)R(K,C)
%%2r(K,C)+R(K,C) & \leq \frac{3}{2}D(K,C)\\
%%\frac{r(K,C)}{R(K,C)} & \geq \frac{D(K,C)}{2R(K,C)}\left(1-\frac{D(K,C)}{2R(K,C)}\right),
%\end{split}
%\]
%i.e. $f(\mathcal K^2,C)\subset f(\mathcal K^2,S)$ for any triangle $S$.
%\end{thm}

Besides showing the validity of the two new inequalities to prove Theorem \ref{thm:dominating_diagram}, we need to see that the boundaries they describe are filled. This is done by the second aspect of our main result.

\begin{thm}\label{thm:BS_diag_simplex}
$f(\mathcal K^2,\mathcal K^2)=f(\mathcal K^2,S)$ for every triangle $S$.
\end{thm}

Notice that Theorem \ref{thm:BS_diag_simplex}  
%proves the validity of the inequalities in Theorem \ref{thm:BS_diag_simplex} when replacing $S$ by any $C\in\mathcal K^2$.
shows that the boundary of $f(\CK^2,\CK^2)$ is completely described from investigating a single diagram, choosing a triangle as the gauge $C$.

Having in mind that $f(\mathcal K^2,S)$ contains every other planar $(r,D,R)$-diagram, if $S$ is a triangle, one may wonder whether the same holds true in the case of Minkowski spaces: does there exist $C_0\in\mathcal K^2_0$ such that
%$f(\mathcal K^2,C_0)$ contains any other diagram for symmetric containers? 
$f(\mathcal K^2,\mathcal K^2_0)=f(\mathcal K^2,C_0)$?
%In this sense, 
We will show
that the right candidate (up to affine transformations) to this covering property would be a \emph{regular hexagon} $H$. 
However, it finally turns out that the answer to this question
in negative, which is a consequence of the following result, which in particular implies that $f(\CK^2,H)$ cannot cover $f(\CK^2,\B_2)$. 

\begin{thm}\label{thm:hexagons}
Let $K\in\mathcal K^2$ and $H\in\mathcal K^2_0$ be such that $H$ is a regular hexagon. If $D(K,H)<2R(K,H)$,
then $r(K,H) \geq R(K,H)/4$.
\end{thm}

Let us mention at this point that it is well known that $D(K,C)=2R(K,C)$ whenever both, $K,C$ are symmetric. Thus, replacing also $\mathcal K^n$ above by $\mathcal K^n_0$ in the first argument gives $f(\mathcal K^n_0,\mathcal K^n_0)=f(\mathcal K^n_0,C) = \conv\left(\set{\left(\begin{smallmatrix} 0 \\ 1 \end{smallmatrix}\right),\left(\begin{smallmatrix} 1 \\ 1 \end{smallmatrix}\right)}\right)$ for all $C \in \CK^n_0$. %and they equal the line segment of endpoints $(0,1)$ and $(1,1)$, for every $C\in\mathcal K^n_0$. 
However, the next result shows that 
%in the general case 
if we do not restrict the first argument this unusual $1$-dimensional behaviour %in $f(\mathcal K^n,C)$ 
only occurs when the gauge $C$ is a parallelotope. 

\begin{thm}\label{thm:squares}
Let $C\in\CK^n$. Then $\conv\left(\set{\left(\begin{smallmatrix} 0 \\ 1 \end{smallmatrix}\right),\left(\begin{smallmatrix} 1 \\ 1 \end{smallmatrix}\right)}\right) \subset f(\CK^n,C)$ and equality holds if and only if $C$ is a parallelotope. 
\end{thm}

Essentially, this is a direct corollary from the well known characterization that parallelotopes are exactly the class of sets with Helly dimension 1 and the theorem even stays true when replacing the inradius by any of the functionals area, perimeter, or (minimal) width. Anyway, since we are not aware of a published proof of the above fact, we will close this gap here.

%\RSays{**This is not really a new \enquote{result}, so maybe we should write something like: \enquote{a well known result from Helly-dimensions can be rephrased in terms of the $(r,D,R)$-diagrams as follows}. Essentially, this is true even more general for $(*,D,R)$-diagrams!**}
%\BSays{Yes, we could say something like that, yepp}

%Finally, we will study further structural properties of the $(r,D,R)$-diagram with respect to other \emph{regular $k$-gons}:
%On the one hand, the diagram of a \emph{square} $Q$ is the only one which is degenerated to a single line segment. We show it in general dimensions.

Theorems \ref{thm:dominating_diagram}, \ref{thm:hexagons}, and \ref{thm:squares} show that Blaschke-Santal\'o diagrams of some $k$-gons, (regular) triangles, regular hexagons and squares) encode extreme behaviours. In the cases of fulldimensional diagrams %\RSays{**if we do not claim \enquote{iff} here, we do not need to distinguish**}\BSays{I do not get this comment},
%\RSays{Doen't matter since we know ask for \enquote{iff}}
they share many common properties. For instance, restricting $K\in\mathcal K^2$ to triangles whenever $C$ is either the euclidean ball $\B_2$ \cite{Ju01}, a triangle \cite{BrKo13}, or a regular hexagon \cite{BrGM17_3}, the supremum over the choices of $K$ of the ratio of the circumradius $R(K,C)$ and the diameter $D(K,C)$ (the so called \emph{Jung constant} of the gauge $C$) is attained if and only if $K$ is an \emph{equilateral} triangle, i.e.~a triangle whose edges all have the same length with respect to $C$. Even so it is easy to argue that the leftmost point in the Jung-boundary must be reached by a triangle it is not always equilateral, as the regular pentagon $C$ shows.
%On the other hand, we show that the \emph{regular pentagon} provides the first example of a Jung extreme triangle which is not equilateral, i.e. such that some of its edges are not diametrical.

\begin{thm}\label{thm:pentagons}
Let $K,P\in\mathcal K^2$ such that $P$ is a regular pentagon. Then
\[
R(K,P) \leq \frac{1+\sqrt{5}}{2}D(K,P).
\]
Moreover, there exist triangles $T$ and $T'$ attaining equality above, such that $T$ is an isosceles triangle, with exactly two diametrical edges, whereas $T'$ is an equilateral one.
\end{thm}
%\RSays{**Somehow one could say \enquote{the pentagon reaches the Jung-boundary to early in the transformation from line to equilateral}**}\BSays{**how could we say that?**}
%\RSays{**I'll explain in a zoom call, but it is not important to note. With your explanation above and the small add on I wrote we should be fine**}

The paper is organized as follows. In Section \ref{sec:preliminaries}, we explain the notations and preliminary results needed along the paper, accompanied by two technical Lemmas.
In Section \ref{sec:parallelotopes} we prove Theorem \ref{thm:squares} using a characterization of parallelotopes and the fact that the intersection of $(r,D,R)$-diagrams equals the diagram of any single parallelotope.
Section \ref{sec:proofsMainResults} treats the main result of the paper, namely, the inequalities exposed in Theorems \ref{thm:dominating_diagram} and \ref{cor:nr+R<=n+1D/2}, and the fact that the union of all $(r,D,R)$-diagrams in the planar case equals the diagram of any single triangle in Theorem \ref{thm:BS_diag_simplex}.
Finally in Section \ref{sec:456}, we show that when restricting to centrally symmetric containers the union of diagrams is no longer given by a single diagram, where the regular hexagon plays a crucial role  (\cf~Theorem \ref{thm:hexagons}). Moreover, we also show that the $(r,D,R)$-diagram of the regular pentagon has a Jung-extreme triangle which is not equilateral (\cf~Theorem \ref{thm:pentagons}).

\section{Notation and preliminary results}\label{sec:preliminaries}

For every $K,C\in\mathcal K^n$ let $K+C=\{x+y:x\in K,\,y\in C\}$ be the \emph{Minkowski sum} of $K$ and $C$ and for every $\lambda \in \R$ let $\lambda K=\{\lambda x :x \in K\}$ be the $\lambda$-dilatation of $K$. We use $-K :=(-1)K$ for short.

For every $X\subset\R^n$, let $\conv(X)$, $\aff(X)$ and $\pos(X)$ be the \emph{convex, affine}, and \emph{positive hull} of $X$, respectively. Moreover, for any $x,y\in\R^n$, let $[x,y]:=\conv(\{x,y\})$ be the \emph{line segment} with endpoints $x$ and $y$.

The circumradius is homogeneous of degree $1$ and monotonically increasing on its first entry, whereas it is homogeneous of degree $-1$ and monotonically decreasing on its second entry, i.e.~for every $K_1,K_2,C_1,C_2\in\mathcal K^n$ with $K_1\subset K_2$, $C_2\subset C_1$, and $\lambda>0$, we have
\[
R(K_1,C_1) \subset R(K_2,C_1) \subset R(K_2,C_2)
%\]
\quad \text{and} \quad
%\[
R(\lambda K_1,C_1)=R(K_1,\frac1\lambda C_1)=\lambda R(K_1,C_1).
\]
Notice that those and the remaining properties of the circumradius are inherited by the inradius and the diameter. This can easily be seen due to the fact that $r(K,C)=R(C,K)^{-1}$ and $D(K,C)=2\max_{x,y\in K}R(\{x,y\},C)$ \cite{BrKo15}.
Notice also that $D(K,C)=D([x,y],C)$ for some \emph{extreme points} $x,y$ of $K$, i.e.~such that no line segment containing $x$ or $y$ on its interior can be contained in $K$ \cite{GrKl}.

For every $K \in \mathcal K^n$, let $\bd(K)$ be the \emph{boundary} of $K$ and if $p\in\bd(K)$ let $N(K,p)$ be the (outer) \emph{normal cone} of $K$ at $p$, i.e.~$N(K,p)=\{u\in\R^n : u^T(x-p) \leq 0 \text{ for all } x \in K\}$.
Moreover, for every $u\in\R^n$ let $h(K,u)$ be the \emph{support function} of $K$ at $u$, defined by $h(K,u) := \sup\{x^Tu : x\in K\}$.
Using the support function, it is well known that the diameter can also be expressed by \[
D(K,C)=2\sup_{u\in\R^n \setminus \set{0}} \frac{h(K-K,u)}{h(C-C,u)}
\] 
for every $K,C\in\mathcal K^n$ \cite{BrKo15}.

We say that $K\subset^{opt}C$ if $K\subset C$ and $K\not\subset c+\lambda C$
for any $\lambda\in(0,1)$ and $c\in\R^n$. The situation in which $K\subset^{opt}C$ is characterized by a touching condition between a finite amount of boundary points of $K$ and $C$ \cite[Theorem 2.3]{BrKo13}.
\begin{prop}\label{prop:opt_cont_criterion}
Let $K,C \in \CK^n$ with $K \subset C$. Then the following are equivalent:
\begin{enumerate}[(i)]
    \item $R(K,C)=1$.
    \item There exist $p^1,\dots,p^k\in K\cap\bd(C)$, for some $k\in\{2,\dots,n+1\}$, 
    and $u^j\in N(C,p^j)$, $j\in[k]$, such that $0\in\conv(\{u^j:j\in[k]\})$.
\end{enumerate}
\end{prop}

In particular, this means that the circumradius is affine invariant, i.e.~for every affine transformation $A$ and $K,C\in\mathcal K^n$ we have $R(A(K),A(C))=R(K,C)$. The same holds for the inradius and diameter.

For the next lemma, one should recognize that using Caratheodory's Theorem we can assume that the points $p^i$ as well as the outer normals $u^i$ in Proposition \ref{prop:opt_cont_criterion} can be chosen affinely independent.

\begin{lem}\label{lem:opt_two_simplices}
Let $K,C \in\CK^n$ be such that $K \subset^{opt} C$. Then
there exist an $\ell$-dimensional simplex $T \in \CK^n$ and
a generalized prism $S \in \CK^n$ with $\ell$-dimensional simplicial base,
such that $T\subset K \subset C\subset S$ for some $\ell \in [n]$, fulfilling
\[
R(T,S)=1,\quad r(T,S) \leq r(K,C),\quad\text{and}\quad D(T,S) \leq D(K,C).
\]
Moreover, if $C=-C$, then we obtain the same conclusion within the chain of inclusions
$T\subset K\subset C\subset S\cap(-S)$, i.e. replacing $S$ by $S\cap(-S)$.
\end{lem}

\begin{proof}
%The proof is a direct corollary of Proposition \ref{prop:opt_cont_criterion}. 
%Indeed,
Since $K \subset^{opt} C$ we know from Proposition \ref{prop:opt_cont_criterion} that there exist $p^1,\dots,p^k \in K\cap\bd(C)$, $k \in\set{2,\dots,n+1}$, and $u^j\in N(C,p^j)$, $j \in [k]$, such that $0 \in \conv(\set{u^1,\dots,u^k})$.
We immediately obtain the claimed result from defining $\ell := k-1$,  
$T:=\conv(\set{p^1,\dots,p^k})$, and $S:=\bigcap_{j=1}^k\set{x\in\R^n : (u^j)^Tx \leq (u^j)^Tp^j}$. 
Moreover, if $C=-C$ then $C\subset S$  directly implies $C\subset -S$, too, and therefore $C\subset S\cap(-S)$.
\end{proof}

%For every $K,C\in\mathcal K^n$, there exists $c\in\mathbb R^n$ such that $c+K\subset R(K,C)C$.
%Taking diameters with respect to $C$ 
%Using the monotonicity 
%and homogeneity\RSays{**really? isn't it only monotonicity?**}
%\BSays{I think you do this: $D(c+K,C)\leq D(R(K,C)C,C)$ (monotonicity), then $D(c+K,C)\leq R(K,C)D(C,C)$ (homogeneity), and finally observe $D(C,C)=2$} 
%\RSays{Now, I see why you had the above observation first. I would argue as follows: $D(K,C) = 2 \max R(\set{x,y},C) \le 2 R(K,C)$}
%\BSays{Yepp, we could go like that too - still, the idea of "monotonicity" is often used after. Maybe it is worth using this "ideas", since they are used below too}
%of $D(\cdot,C)$ applied above we directly deduce from the definition of the diameter that
Just combining the definition of the diameter and the monotonicity of the circumradius we directly deduce that
\begin{equation}\label{eq:D<=2R}
D(K,C) \leq 2\,R(K,C).
\end{equation}
Equality in \eqref{eq:D<=2R} holds, for instance, whenever $K=(1-\lambda)[x,y]+\lambda C$, for some $x,y\in\mathbb R^n$ and $\lambda\in[0,1]$. Seeking for a reverse inequality to the one above, led to the \emph{Jung constants} 
\[
j_C=\sup\{R(K,C)/D(K,C):K\in\mathcal K^n\},
\]
referring to Jung who showed that
$j_{\mathbb B_2}=\sqrt{n/(2(n+1))}$ \cite{Ju01}. Bohnenblust proved
that 
\begin{equation}\label{eq:Bohnenblust}
j_C\leq n/(n+1)
\end{equation}
whenever $C\in\mathcal K^n_0$ \cite{Bo38}. Moreover,
one has $R(K,C)=n/(n+1)D(K,C)$ if and only if $K$ is an $n$-dimensional \emph{simplex},
i.e.~the convex hull of $n+1$ affinely independent points, with barycenter $0$ 
%\RSays{**since writing \enquote{iff} above we cannot do assumptions on the position of $K$ here, but need to introduce a translate in the containment chain below**}\BSays{Yepp}
and $C$ fulfills 
\begin{equation}\label{eq:equalityBohnenblust}
    K-K \subset D(K,C)C \subset (n+1)(K\cap(-K))
\end{equation} 
(see \cite[Corollary 2.9]{BrGM17_3}). More generally, we also know that $j_C\leq n/2$, with $2R(K,C)=nD(K,C)$ if $K=-C$ is a fulldimensional simplex \cite{BrKo13}.

Maybe, the first non trivial inequality relating all three functionals was the so called \emph{concentricity inequality}
(proven in \cite{Sa61} for the euclidean plane, in \cite{Vr81} for general euclidean spaces, and in \cite{MoPaPh} in the general case), 
which states that for any $K\in\mathcal K^n$ and $C\in\mathcal K^n_0$ we have
\begin{equation}\label{eq:r+R<=D}
r(K,C)+R(K,C) \leq D(K,C).
\end{equation}
For every $C\in\CK^n$, the \emph{asymmetry measure of Minkowski} $s(C)$, or \emph{Minkowski asymmetry} for short, is the smallest $\lambda \ge 0$ such that $\lambda C$ contains some translate of $-C$, i.e.~$s(C)=R(-C,C)$ (see \cite{Gr36}).
It is well known that $s(C) \ge 1$, with equality if and only if $C \in\CK^n_0$, and $s(C) \le n$, with equality if and only if $C$ is a fulldimensional simplex. 
%Moreover, we say that $C$ is \emph{Minkowski centered} if $C\subset-s(C)$.
%This functional has been well studied, see, e.g.,~\cite{Gr36}.

Making use of the Minkowski asymmetry the concentricity inequality has been generalized for arbitrary $C\in\CK^n$ \cite{BrGM17_2}:   
%For every $K,C\in\mathcal K^n$, the generalized concentricity inequality states that
\begin{equation}\label{eq:sr+R<=s+1D/2}
s(C)r(K,C)+R(K,C) \leq \frac{s(C)+1}{2}D(K,C).
\end{equation}
In particular, when $s(C)=n$, equality holds for sets of the form $K=(1-\lambda)(-C)+\lambda C$, $\lambda\in[0,1]$ (and, more generally, for constant width sets with respect to $C$)%\BSays{**be careful, this statement in the latter parenthesis is false - remember that the Reuleaux triangle does not fulfill equality here when the container is the Reuleaux pentagon**}\RSays{**This is your comment, not mine ;-) -- I thought we said it is correct as long as $C$ is a simplex.**}\BSays{**ugh, nevermind!**}.

%\RSays{**Since the next lemma is againg true in general dimensions, I would suggest to move it to Section 2. Moreover, if we do, we can have the proof of \eqref{eq:New_Ineq_isosceles_triangle} directly within the proof of the rest of Theorem \ref{thm:dominating_diagram} directly after it. Moreover, we should cite the precursor results (like the one in our 3D-paper?)**}
%\BSays{**Yes, I agree too - specially after the comments above**}
In the Euclidean case it is shown in \cite[Lemma 2.1]{BrGM17}) that $f(\CK^n,\B_n)$ is star-shaped with respect to the upper-right vertex $f(\B_n,\B_n)=(1,1)^T$. For practical purposes this means that these diagrams can be fully described by simply explaining the sets mapped onto its boundaries.
The next lemma shows that the star-shapedness is still true when replacing $\B_n$ by an arbitrary $C\in\CK^n$.
\begin{lem}\label{lem:starshaped}
Let $K,C\in\mathcal K^n$. 
%\RSays{**This isn't necessary, it is just \enquote{wlog}**}
%\BSays{**For some reason, I had a hard time justifying it :D this is why i left it as it is now..**}. 
Then
\[
f((1-\lambda)K+\lambda C,C)=(1-\lambda)f(K,C)+\lambda f(C,C),
\]
for every $\lambda\in[0,1]$.
\end{lem}

\begin{proof}
We start noticing that because of the translation and dilatatation invariance of $f$ we may assume without loss of generality that $K\subset^{opt}C$ and doing so we have
$(1-\lambda)K+\lambda C \subset (1-\lambda)C+\lambda C = C$. %from which follow that $R((1-\lambda)K + \lambda C,C)\leq 1$. Moreover, by 
Using Proposition \ref{prop:opt_cont_criterion}, there exist $p^1,\dots,p^k \in K \cap \bd(C)$ and $u^j\in N(C,p^j)$, $j\in[k]$ for some $k \in\{2,\dots,n+1\}$, such that $0 \in \conv(\set{u^1,\dots u^k})$. 
Now, since $p^j = (1-\lambda)p^j+\lambda p^j \in ((1-\lambda)K + \lambda C) \cap \bd(C)$ we can again conclude from Proposition \ref{prop:opt_cont_criterion} that $R((1-\lambda)K+\lambda C,C)=1$, for every $\lambda\in[0,1]$. 

The same argument give us $r((1-\lambda)K+\lambda C,C)=(1-\lambda)r(K,C)+\lambda$ and
$D((1-\lambda)K+\lambda C,C)=(1-\lambda)D(K,C)+2\lambda$. Thus
\[
\begin{split}
f((1-\lambda)K+\lambda C,C) & =\left((1-\lambda)r(K,C)+\lambda,(1-\lambda)\frac{D(K,C)}{2}+\lambda\right) \\
& = (1-\lambda)f(K,C)+\lambda (1,1),
\end{split}
\]
concluding the proof.
\end{proof}

\section{The Helly-dimension and its meaning for the minimal diagram}\label{sec:parallelotopes}

The \emph{Helly dimension} $\him(C)$ of a set $C\in\mathcal K^n$ is defined as the smallest positive number $k\in\N$ such that whenever we consider a set of indices $I\subset\N$ with the property $\bigcap_{i\in J}(x_i+C) \neq \emptyset$ for all $J\subset I$ with $|J|\leq k+1$ and  $x_i\in\R^n$, it already follows that $\cap_{i\in I}(x_i+C)\neq \emptyset$ (for more details on the Helly-dimension see \cite{BaKaPa, SzNa}). We say that a point $p\in\bd(C)$, $C\in\mathcal K^n$, is \emph{regular} or \emph{smooth} if $\dim(N(C,p))=1$. 
In \cite[Ch. IV]{BoMaSo} it is proven that the Helly dimension is equivalent to
the \emph{minimal dependence} $\mathrm{md}(C)$, which is the largest number $k\in\N$ such that there exist regular points $p^j\in \bd(C)$, $j\in[k+1]$, and
vectors $u^j\in N(C,p^j)$, $j\in[k+1]$, such that $0\in\conv(\{u^j:j\in[k+1]\})$ and such that for every
$I\subset[k+1]$, $|I|\leq k$, the vectors $\{u^j:j\in I\}$ are linearly independent.

Finally, notice that Sz\"okefalvi-Nagy \cite{SzNa} proved that if $C\in\CK^n$ then $\him(C)=1$ if and only if $C$ is a parallelotope.

%\BSays{**In the proof below we literally do ``$R=R_1$ implies $m_d=1$'' and ``$C$ parallelotope implies $R=R_1$''. $him$ is actually not used at all...Btw, I brought everything here from below, but I did not change anything else.**}

\begin{lem}\label{lem:parallelotopes}
Let $C\in\CK^n$. The following are equivalent:
\begin{itemize}
    \item[(i)] $C$ is a parallelotope.
    \item[(ii)] $D(K,C)=2R(K,C)$ for every $K\in\CK^n$.
\end{itemize}
\end{lem}

\begin{proof}

Let us first mention that we make use of the fact $\him(C) = \mathrm{md}(C) =1$ if and only if $C$ is a parallelotope and essentially prove that $(ii)$ implies $\mathrm{md}(C)=1$, while $C$ being parallelotope implies $(ii)$.

We start proving \enquote{(ii) $\Rightarrow$ (i)} and assume that $k=\mathrm{md}(C)>1$. That means there exist smooth boundary points $p^1,\dots,p^{k+1}$ of $C$ and $u^j \in N(C,p^j)$, $j \in [k+1]$, such that $0 \in \conv(\set{u^1,\dots,u^{k+1}})$ and since $\mathrm{md}(C) = k$ we know that $\{u^{j_1},\dots,u^{j_k}\}$ is linearly independent for every choice $1\leq j_1<\cdots<j_k\leq k+1$. 

Now, let $K=\conv(\set{p^1,\dots,p^{k+1}})\in\CK^n$ and notice that by Proposition \ref{prop:opt_cont_criterion} we have $K\subset^{opt} C$.
Using (ii) and assuming without loss of generality that $p^1,p^2$ is a diametrical pair of $K$, we conclude $2=2R(K,C)=D(K,C)=D([p^1,p^2],C)$. Hence we have $R([p^1,p^2]C)=1$ and therefore $[p^1,p^2] \subset^{opt} C$. %Let $x,y\in K$ be such that $[x,y]$ is a diametrical segment of $K$, i.e. such that $D([x,y],C)=D(K,C)$. Since $K=\mathrm{conv}(\{p^1,\dots,p^{k+1}\})$, let us suppose without loss
%of generality that $x=p^1$ and $y=p^2$. Using \eqref{eq:D<=2R} and the monotonicity of $R$, therefore 
%\[
%2=D(K,C)=D([p^1,p^2],C)\leq 2R([p^1,p^2],C)\leq 2R(K,C)=2,
%\]
%i.e. $R([p^1,p^2],C)=1$ with $[p^1,p^2]\subset C$. 
%Again by Proposition \ref{prop:opt_cont_criterion},
%there exist $v^j\in N(C,p^j)$, $j=1,2$, such that $\beta_1v^1+\beta_2v^2=0$, for some $\beta_j\geq 0$. Evidently, since
%$p^j$ are regular within $C$, $v^j=u^j$, $j=1,2$. In particular, we can conclude that $u^2=-u^1$. 
%Thus, in particular $\{u^1,u^2\}$ are linearly dependant, contradicting the fact that $\mathrm{md}(C)>1$, and hence concluding the if part.
Using the smoothness of $p^1,p^2$ and Proposition \ref{prop:opt_cont_criterion} again, we conclude $0 \in \conv(\set{u^1,u^2})$, which shows $\mathrm{md}(C)=k=1$.

We now show \enquote{(i) $\Rightarrow$ (ii)} and remember that (i) implies that $C$ is a parallelotope. 
After suitable translations and dilatations of $K$ and $C$ we may assume $C=\bigcap_{i=1}^n\set{x\in\R^n:|(v^i)^Tx| \leq 1}$ for some
linearly independent vectors $v^1,\dots,v^n\in\R^n$ as well as $K \subset^{opt} C$. By Proposition \ref{prop:opt_cont_criterion} there exist 
$p^1,\dots,p^{n+1}\in K\cap\bd(C)$ and $u^j\in N(C,p^j)$, $j \in [n+1]$, such that $0 \in \conv(\set{u^1,\dots,u^{n+1})}$. 
To fulfill the latter the points $p^1,\dots,p^{n+1}$ have to touch a pair of opposing facets of $C$. Otherwise (after changing signs of the $v^i$, if necessary) %$\{x\in C:x^Tv^j=-1\}$, $j=1,\dots,n$,
we would have $(v^i)^Tp^j =1$ for all $j \in [n+1]$ and $i \in [n]$, which directly implies $u^j \in \pos(\set{v^1,\dots,v^n})$ for all $j \in [n+1]$, contradicting $0 \in \conv(\set{u^1,\dots,u^{n+1})}$. 

Thus, we may assume $(v^1)^Tp^1=-1$ and $(v^1)^Tp^2=1$. However, this implies by Proposition \ref{prop:opt_cont_criterion} that
$[p^1,p^2] \subset^{opt} C$ 
%from which
%\[
%2R([p^1,p^2],C)=D([p^1,p^2],C)\leq D(K,C) \leq 2R(K,C)=2=2R([p^1,p^2],C),
%\]
and therefore $D(K,C) \ge 2R([p^1,p^2],C) = 2R(K,C) \ge D(K,C)$, which concludes the proof.
\end{proof}

\begin{proof}[Proof of Theorem \ref{thm:squares}]
The first part is a direct consequence of $f(K,C) = \left(\begin{smallmatrix} 0 \\ 1 \end{smallmatrix} \right)$ if $K$ is a segment and $f(C,C) =  \left(\begin{smallmatrix} 1 \\ 1 \end{smallmatrix}\right)$ in combination with Lemma \ref{lem:starshaped}, while the second part follows directly from Lemma  \ref{lem:parallelotopes}.
\end{proof}

%\begin{thm}
%Let $K,C\in\mathcal K^2$. Then
%\begin{equation}\label{eq:Isosceles_triangle}
%\frac{r(K,C)}{R(K,C)} \geq \frac{D(K,C)}{2R(K,C)}\left(1-\frac{D(K,C)}{2R(K,C)}\right).
%\end{equation}
%Moreover, when $D(K,C)<2R(K,C)$ equality holds if $K$ and $C$ are triangles such that 
%a diametrical edge of $K$ is parallel to an edge of $C$; when $D(K,C)=2R(K,C)$ equality holds if $K$ is a line segment.
%\end{thm}

%\begin{thm}
%Let $K,S\in\mathcal K^2$ be such that $S$ is a triangle. Then \eqref{}, \eqref{}, and \eqref{} form a complete system of %inequalities for the inradius, the circumradius, and the diameter measured with respect to $S$. In particular, the diagram %$f(\mathcal K^2,S)$ is fully described by those three inequalities.
%\end{thm}

%\begin{cor}\label{cor:diagram_general_contained_simplex}
%Let $C,S\in\mathcal K^2$ be such that $S$ is a triangle. Then $f(\mathcal K^2,C)\subset f(\mathcal K^2,S)$.
%\end{cor}

\section{The maximal diagram -- proof of the main result}
\label{sec:proofsMainResults}

We start noticing that Theorem \ref{cor:nr+R<=n+1D/2} is a direct consequence of \eqref{eq:sr+R<=s+1D/2}. 
%which \RSays{proves that the second inequality given in Theorem \ref{thm:BS_diag_simplex} is valid for $f(\CK^2,\CK^2)$.}
%gives one of the two new inequalities involved in Theorem \ref{thm:dominating_diagram}. 
%\RSays{**Theorem 1.4 above is given for general $n$ even though we concentrate on $n=2$. Why don't we state Cor. 3.1 as its own theorem and obtain the $n=2$ inequality from it as a corollary?}

%\BSays{REMEMBER!**Notice that Theorem \ref{thm:dominating_diagram} %is equivalent to 
%proves the validity of the inequalities in Theorem \ref{thm:BS_diag_simplex} when replacing $S$ by any $C\in\mathcal K^2$.**}

\begin{proof}[Proof of Theorem \ref{cor:nr+R<=n+1D/2}]
Let $g(x):=ax/(x+1)+b/(x+1)$ for some $0\leq a\leq b$ and $x\geq 1$. Since
$g'(x)=(a-b)/(x+1)^2\leq 0$ we see that $g$ is non increasing.  
Now, using \eqref{eq:sr+R<=s+1D/2}, the fact that $0 \leq r(K,C) \leq R(K,C)$, and that $s(C) \leq n$, we conclude that
\[
\frac{n}{n+1}r(K,C)+\frac{1}{n+1}R(K,C) 
\leq \frac{s(C)}{s(C)+1}r(K,C)+\frac{1}{s(C)+1}R(K,C)
\leq \frac{D(K,C)}{2}.
\]
\end{proof}

Now, it only remains to prove the validity and tightness of the third inequality in Theorem \ref{thm:dominating_diagram}. To do so we should first recognize that by dividing through the circumradius $R(K,C)$ this inequality can be rewritten as 
%\begin{thm}\label{thm:New_Ineq_isosceles_triangle}
%Let $K,C\in\mathcal K^2$. Then
\begin{equation}\label{eq:New_Ineq_isosceles_triangle}
\frac{r(K,C)}{R(K,C)} \geq \frac{D(K,C)}{2R(K,C)}\left(1-\frac{D(K,C)}{2R(K,C)}\right)
\end{equation}
or simply $D(K,C)(2-D(K,C)) \le 4r(K,C)$ in case $R(K,C)=1$.

%\end{thm}

First we show that certain triangles attain equality in that inequality.

\begin{lem}\label{lem:Isosceles_Extreme_Triangle}
Let $S=\conv(\{p^1,p^2,p^3\}) \in\mathcal K^2$ be an equilateral triangle,
with $\|p^i\|=1$, $i=1,2,3$, and $(p^i)^Tp^j=-1/2$, $1\leq i<j\leq 3$. Moreover, let
\[
T = \conv\left(\left\{ \frac12(p^1+p^2), \frac{D}{2}p^1+\left(1-\frac{D}{2}\right)p^3,\frac{D}{2}p^2+\left(1-\frac{D}{2}\right)p^3\right\}\right),
\]
for some $D\in[1,2]$. Then $R(T,S)=1$, $D(T,S)=D$, and
\[
r(T,S)=\frac{D(T,S)}{2}\left(1-\frac{D(T,S)}{2}\right).
\]
\end{lem}

\begin{proof}
First of all, recognize that if $D=2$ then $T=[p^1,p^2]$, which directly means that 
$D(T,S)=2R(T,S)=2$ and $r(T,S)=0$, proving the claim in that case.

Hence we may suppose $D<2$. In this case each vertex of $T$ touches 
each of the edges of $S$ in their relative interior, which by Proposition \ref{prop:opt_cont_criterion} directly implies $R(T,S)=1$. 

Now, observe that since $D \ge 1$ we have 
\[
\frac{D}{2}p^i+\left(1-\frac{D}{2}\right)p^3 \in \left[p^i,\frac12(p^i+p^3)\right],
\]
for $i=1,2$,
%\RSays{**Essentially the third point is not needed. Why listing?**},
hence implying that
\[
\begin{split}
D\left(\left\{\frac{D}{2}p^i+\left(1-\frac{D}{2}\right)p^3,\frac12(p^1+p^2)\right\},S\right) &
\leq D\left(\mathrm{conv}\left(\left\{p^i,\frac12(p^i+p^3),\frac12(p^1+p^2)\right\}\right),S\right) \\
& = D\left(\left\{p^i,\frac12(p^1+p^2)\right\},S\right) \\
& = \frac{D(\{p^1,p^2\},S)}{2} = 1\\
& \leq D = D\left(\left\{\frac{D}{2}p^j+\left(1-\frac{D}{2}\right)p^3:j=1,2\right\},S\right)
%& = D\left(\left\{\frac12(p^1+p^3),\frac12(p^2+p^3)\right\},S\right)\\
%& \leq D\left(\left\{\frac{D}{2}p^j+\left(1-\frac{D}{2}\right)p^3:j=1,2\right\},S\right) \\
%& = \frac{D}{2}\cdot D(\{p^1,p^2\},S)=D,
\end{split}
\]
for both $i=1,2$, thus showing that $D(T,S)=D$.
%\RSays{**I mentioned this the last time and I still do not get the point here: $\conv\set{p^i,\frac{1}{2}(p^1+p^2)}$ is half the segment of $\conv\set{p^1,p^2}$. So we immediately have $R(\conv\set{p^i,\frac{1}{2}(p^1+p^2)},C) = \frac12 R(\conv\set{p^1,p^2},C)$, whatever $C$ is. Where should be the mistake?**}

We now compute the vertices $z^i:=c+r(T,S)p^i$, $i=1,2,3$, of $c+r(T,S)S$ such that $c+r(T,S)S \subset T$, for some $c\in\mathbb R^2$. 
Without loss of generality, we may assume $p^1=(\sqrt{3}/2,-1/2)^T$, $p^2=(-\sqrt{3}/2,-1/2)^T$, and $p^3=(0,1)^T$.
The symmetry of $T$ and $S$ with respect to the vertical line passing through the origin implies that 
\[
\begin{split}
z^3 & =\frac12\left(\frac D2\left(\frac{\sqrt{3}}{2},-\frac12\right)^T+\left(1-\frac{D}{2}\right)(0,1)^T\right)+\frac12\left(\frac D2\left(-\frac{\sqrt{3}}{2},-\frac12\right)^T+\left(1-\frac{D}{2}\right)(0,1)^T\right) \\
& = \left(0,1-\frac{3D}{4}\right)^T.
\end{split}
\]

%Let $z^1$ be the vertex of $c+r(T,S)S$ belonging to the edge of $T$ of vertices $(D/2)(\sqrt{3}/2,-1/2)+(1-\frac{D}{2})(0,1)$ and $(1/2)(p^1+p^2)=(0,-1/2)$. Thus, 
Let $t\geq 0$ be such that
\[
z^1=z^3+t(p^1-p^3)=\left(0,1-\frac{3D}{4}\right)^T+t\left(\frac{\sqrt{3}}{2},-\frac32\right)^T=
\left(\frac{\sqrt{3}t}{2},1-\frac{3D}{4}-\frac{3t}{2}\right)^T.
\]
In particular, we get 
\[
r(T,S)=\frac{t\|p^1-p^3\|}{\|p^1-p^3\|}=t.
\]
Notice, the line containing the vertices $(D/2)(\sqrt{3}/2,-1/2)^T+(1-D/2)(0,1)^T$
and $(0,-1/2)^T$ of $T$ is described (in $(x,y)$ coordinates) by the equation
\[
y = \frac{1-\frac{3D}{4}+\frac12}{\frac{\sqrt{3}D}{4}}x - \frac12 = \frac{6-3D}{\sqrt{3}D}x - \frac12.
\]
Since $z^1$ has to fulfill this equation, we obtain
\[
1 - \frac{3D}{4} - \frac{3t}{2} = \frac{(6-3D)t}{2D} - \frac{1}{2}1 - \frac{3D}{4} - \frac{3t}{2} = \frac{(6-3D)t}{2D} - \frac{1}{2} 
%\frac{2t}{D}\left(\frac32-\frac{3D}{4}\right),
\]
which shows $r(T,S)=t=D(2-D)/4$, concluding the proof.
\end{proof}

\begin{lem}\label{lem:New_ineq_isos_triangles_Triang}
Let $T,S\in\CK^2$ be both triangles. Then
%, such that $S$ is an equilateral one \RSays{**equilateral can be removed here, or? It suffice to say \enquote{wlog $S$ is an equilateral} in the proof**}\BSays{**see Remark 3.3. Shall we integrate it here??**}\RSays{I would suggest so. For me it is just a little \enquote{wlog-comment} and I don't think we need it that big as Remark 3.3}. Then
\begin{equation}\label{eq:triangles_and_triangles}
\frac{r(T,S)}{R(T,S)} \geq \frac{D(T,S)}{2R(T,S)}\left(1-\frac{D(T,S)}{2R(T,S)}\right).
\end{equation}
\end{lem}

\begin{proof}
Using the affine invariance of the radii, and since affine transformations of simplices are simplices, we can assume
without loss of generality that $S$ is an equilateral triangle centered at the origin.
In particular, let $S=\conv(\{p^1,p^2,p^3\})$ and $T=\conv(\{q^1,q^2,q^3\})$, for some $p^i,q^i\in\R^2$ with $\|p^i\|=1$, $i=1,2,3$.
Remember that $r(T,S)$, $D(T,S)$, and $R(T,S)$ are all homogeneous of degree $1$ on $T$. Hence \eqref{eq:triangles_and_triangles} holds true for 
$T$ if and only if \eqref{eq:triangles_and_triangles} holds true for 
$T/R(T,S)$. Thus we may assume that $R(T,S)=1$ and moreover, using an additional translation if necessary, that $T \subset^{opt} S$.
Using Proposition \ref{prop:opt_cont_criterion} this again allows us to assume that each $q^i$ belongs to the edge of $S$ opposing $p^i$, $i=1,2,3$.
Furthermore, since $D(T,S)$ is attained by two extreme points, we can suppose that $D(T,S)=D([q^1,q^2],S)$. Finally, without loss of generality, we may assume that $\|q^2-p^1\|\leq\|q^1-p^2\|$. 
Keeping the equilaterality of $S$ in mind this directly implies that if we choose $t \in \R^2$, such that $t+[q^1,q^2]\subset^{opt}R([q^1,q^2],S)S$ then $t+q^2 = R([q^1,q^2],S)p^1$.
%\BSays{A basic geometric argument shows that we get $c+[q^1,q^2]\subset^{opt}R([q^1,q^2],S)S$, for some $c\in\R^2$, such that $c+[q^1,q^2]$ touches necessarily the vertex $R([q^1,q^2],S)p^1$.}
Using this as well as the fact that $-p^1$ is an outer normal of the edge $[p^2,p^3]$ of $S$ we see that there exists a pair of parallel lines orthogonal to $p^1$ that supports $t+[q^1,q^2]$ as well as $R([q^1,q^2],S)S$ which in particular implies
\[
2\frac{h([q^1-q^2,q^2-q^1],p^1)}{h(S-S,p^1)}=D([q^1,q^2],S)=D(T,S).
\]
%\RSays{**However, I do not see exactly why we want to go via the breath here. **} \BSays{**We do it for some arguments below. For instance, right after this lines, the reason why $T$ has to be contained in such set - which is used afterwards - is a direct consequence of the diameter attaining in the direction $p^1$**}
%\RSays{**Can't we argue more directly that $\max_{x \in T}(p^1)^Tx = (p^1)^Tq^2$ from the argument before? And isn't it that what we are looking for?**}
%\BSays{**Yes, what you wrote is the consequence we actually use below, i.e. what we use afterwards. **}

In particular, $T \subset \set{x \in S : (p^1)^Tx \leq (p^1)^Tq^2}$ and, using the symmetry of $S$ with respect to the line through 0 in direction of $p^3$, 
%and since $2h([p^1-p^2,p^2-p^1],p^2)/h(S-S,p^2) \leq D([q^1,q^2],S)$ 
it must also be true that $T \subset \set{x\in S: x^Tp^2 \leq (p^1)^Tq^2}$. Otherwise the width of $T$ in direction of $p^2$ would be greater than in direction of $p^1$, contradicting that the latter is the maximal width.

%\BSays{**Like that (see image)**}
%\RSays{**Exactly -- now I only would switch to $(p^i)^Tx$ to have the same order than the right hand side.**}

\begin{figure} \label{fig:transform-T}
\begin{center}
    \includegraphics[width= 10cm]{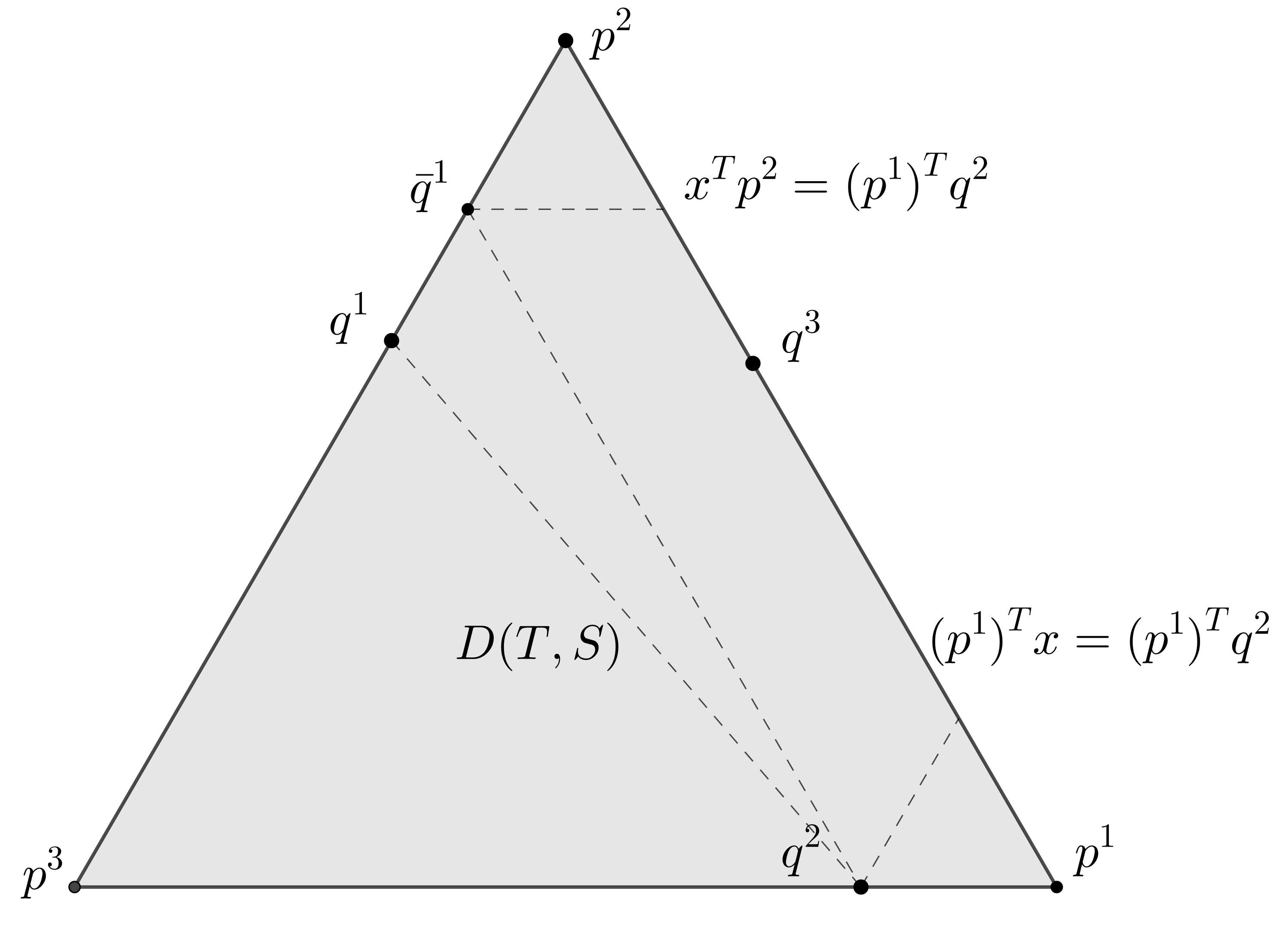}
\end{center}
\caption{Points $q^i$, $i\in[3]$, over the edges of the triangle $S=\mathrm{conv}(\{p^1,p^2,p^3\})$, and the point $\bar q^1$ which
replaces $q^1$ at the end of the first step of the proof.}
\end{figure}

For the next step, let $c\in\R^2$ be such that $c+r(T,S)S \subset T$. The idea now is that we replace $q^1$ by the only point $\bar q^1$ in $[p^2,p^3]$ such that $[\bar q^1,q^2]$ becomes parallel to $[p^1,p^2]$.
Using the intercept theorem we obtain $D([\bar q^1,q^2],S)=D([q^1,q^2],S)$ as well as $(p^1)^Tq^2 = (p^2)^T \bar q^1 \ge (p^2)^Tq^3$ and therefore
$D([\bar q^1,q^3],S) \leq D([q^1,q^3],S)$. Moreover, since $q^3 \in T \subset \set{x \in S : (p^1)^Tx \leq (p^1)^Tq^2}$ we have $(p^1)^Tq^3 \le (p^1)^Tq^2$. Hence 
\[
d(\bar q^1,\aff(q^2,q^3)) \leq d(q^1,\aff(q^2,q^3)),
\]
%\BSays{**Emm I think we actually need to argue with the fact that $q^3\in T \subset \set{x \in S : (p^1)^Tx \leq (p^1)^Tq^2}$; otherwise, the statement is false (simply, draw $q^3$ beyond that region and you will see that those distances are switching - and the argument would fail**}
%\RSays{**Good the we have the Fig now next to it -- now I got your argument: the problem is the region behind the dashed line opposing the edge of $S$ including $q^1$, not that one intersecting it. You are totally right.**}
where $d(\cdot,\cdot)$ denotes the Euclidean distance. 
%\RSays{**This should be moved to Section 2.**}\BSays{**Maybe I am saying bullshit :D but it may perfectly fit here, since we use it only here, and we introduced the euclidean norm in the introduction and not in section 2 - thus, writing it in section 2 stays quite strange and lonely, somehow.**}\RSays{No, you are totally right. Since introducing the euclidean norm only in Sec1 and not Sec2 it fits better here.}
Let $L:=\aff([c+r(T,S)p^2,c+r(T,S)p^3])$ and $\bar T := \conv(\set{\bar q^1,q^2,q^3})$. The inequality above shows that the length of the segment $T\cap L$ is not shorter than the length of the segment $\bar T \cap L$.
Moreover, since the distance from $q^3$ to $L$ is not longer than the distance from $q^2$ to $L$, we can conclude that we cannot find a triangle $\bar c+\rho S$ contained in $\bar T$ with $\rho > r(T,S)$ (and $\rho = r(T,S)$ can only happen if $q^1=\bar q^1$ or if $[q^2,q^3]$ is parallel to $[p^2,p^3]$). Thus, we have shown that $D(\bar T,S) = D(T,S)$ and $r(\bar T,S) \leq r(T,S)$, and evidently, we also have $R(\bar T,S)=1$.

Finally, if we consider any triangle $\widetilde{T} := \conv(\{q^1,q^2,\widetilde{q}^3\})$ with $\widetilde{q}^3$ in $[p^1,p^2]$, such that
$\|\widetilde{q}^3-p^1\| \geq \|q^2-p^1\|$ and $\|\widetilde{q}^3-p^2\| \geq \|q^1-p^2\|$ we have $D(\widetilde{T},S)=D(T,S)$ for the same reasons than above and surely we also have $R(\widetilde{T},S)=1$ again. Moreover, in case of $[q^1,q^2]$ parallel to $[p^1,p^2]$ we obtain again from the intercept theorem that
$r(\widetilde{T},S) = r(T,S)$.
Thus we may select $\widetilde{q}^3 := \frac12(p^1+p^2)$ and obtain from the above
that the inradius of any triangle $T$ of circumradius $1$ and diameter $D$ (all with respect to $S$), is at least as big as the inradius of the final triangle $\widetilde{T}$ and it is an isosceles triangle by the choice of $q^3$. Thus we are in the situation described in Lemma \ref{lem:Isosceles_Extreme_Triangle}. Hence
\[
r(T,S) \geq r(\widetilde{T},S) = \frac{D(\widetilde{T},S)}{2}\left(1-\frac{D(\widetilde{T},S)}{2}\right)
= \frac{D(T,S)}{2}\left(1-\frac{D(T,S)}{2}\right),
\]
concluding the result.
\end{proof}

%\begin{rem}\label{rmk:true_every_triangle}
%The acquainted reader has surely realized that Lemma \ref{lem:New_ineq_isos_triangles_Triang} 
%is proven for $S$ an equilateral triangle. Let us denote for simplicity $F(f(T,S))\geq 0$ the 
%inequality \eqref{eq:triangles_and_triangles} in Lemma \ref{lem:New_ineq_isos_triangles_Triang}, for some
%continuous $F:[0,1]^2\rightarrow\mathbb R$, which holds true for every triangle $T$ and an equilateral triangle $S$.
%If we are now given any other triangle $S'$, there exists some affine transformation $A$
%such that $A(S')=S$. By the affine invariance of $f(T,S)$, we thus have that 
%$f(T,S)=f(A(T),A(S))$. Finally, we also have by Lemma \ref{lem:New_ineq_isos_triangles_Triang} that
%\[
%F(f(T',S'))=F(f(A(T),A(S)))=F(f(T,S))\geq 0
%\]
%holds true for every $T'=A(T)$, $T$ any triangle. Since $T$ ranges all over the set of triangles, thus $A(T)$ also ranges over the set of %all triangles, showing that Lemma \ref{lem:New_ineq_isos_triangles_Triang} holds true when replacing $S$ by
%any other triangle.
%\end{rem}

%   \begin{figure}[h]
%      \includegraphics[trim = 0cm 3cm 0cm 3cm, width=7cm]{propRrDTriangles.pdf}
%      \includegraphics[trim = 0cm 3cm 0cm 3cm, width=7cm]{theoRrDTriangles.pdf}
%            \caption{}\label{1}
%  \end{figure}

%Before showing Theorem \ref{thm:New_Ineq_isosceles_triangle}, we need to establish a consequence of Proposition \ref{prop:opt_cont_criterion}.

%\begin{proof}[Proof of Theorem \ref{thm:New_Ineq_isosceles_triangle}]

\begin{proof}[Proof of Theorems \ref{thm:dominating_diagram} and \ref{thm:BS_diag_simplex}]
%\BSays{**Maybe we should start now saying in a single sentence that the three inequalities in Theorem \ref{thm:dominating_diagram}
%are valid (i.e. \eqref{eq:D<=2R}, \eqref{eq:nr+R<=n+1D/2}, and \eqref{eq:New_Ineq_isosceles_triangle}).**}
%\RSays{**Have a look:**}
Since we already know the general validity of  \eqref{eq:nr+R<=n+1D/2} and \eqref{eq:D<=2R}, we start proving the general validity of \eqref{eq:New_Ineq_isosceles_triangle}:
After a suitable dilatation, we can assume that $R(K,C)=1$ and in case $D(K,C)=2$, the inequality is obviously true. Thus we may assume $D(K,C)<2$ in the following. 
Let $T,S \subset\CK^2$ be as given in Lemma \ref{lem:opt_two_simplices}, such that  $R(T,S)=1$, $r(T,S)\leq r(K,C)$, and $D(T,S)\leq D(K,C)$. Recognize that $T$ is a segment and $S$ a rectangle if $\ell = 1$ and since they are both symmetric we immediately obtain $D(K,C)=2$, contradicting our assumption.
%$K$ and $C$ only share two touching points $p^1,p^2\in K\cap\bd(C)$ and the two outer normals $u^j\in N(C,p^j)$, $j=1,2$, with the property $0\in\mathrm{conv}(\{u^1,u^2\})$ have to fulfill $u^1=-u^2$, from which $D(K,C)=2$, a contradiction. 

Thus, we must have $\ell=2$, i.e., $T$ and $S$ are both triangles.

For the sake of contradiction, let us assume that 
\[
r(K,C)< \frac{D(K,C)}{2}\left(1-\frac{D(K,C)}{2}\right).
\]
Let $g(y)=(y/2)(1-y/2)$, $y\in(1,2)$, and notice that $g'(y)=(2-2y)/4<0$. 
Thus $g(y)$ is a strictly decreasing function whenever $y\in(1,2)$.
Hence we would obtain
\[
r(T,S)\leq r(K,C) < \frac{D(K,C)}{2}\left(1-\frac{D(K,C)}{2}\right) \leq \frac{D(T,S)}{2}\left(1-\frac{D(T,S)}{2}\right),
\]
contradicting Lemma \ref{lem:New_ineq_isos_triangles_Triang}.
%together with Remark \ref{rmk:true_every_triangle}. 
%Thus
%\[
%r(K,C)\geq \frac{D(K,C)}{2}\left(1-\frac{D(K,C)}{2}\right).
%\]
%\BSays{Evidently, \eqref{eq:D<=2R} and \eqref{eq:nr+R<=n+1D/2} (see Theorem \ref{cor:nr+R<=n+1D/2}) hold true as well.}

To finish the proof we know show that the boundaries described by the three inequalities are filled. First, notice that $(1-\lambda)[x,y]+\lambda S$, $x,y\in\R^2$,
attains equality in \eqref{eq:D<=2R} for all $\lambda\in[0,1]$ and in particular, the line segment $[x,y]$ and $S$ are the extreme cases $\lambda \in \set{0,1}$.
Second, we have $(1-\lambda)(-S)+\lambda S$, $\lambda\in[0,1]$, attains equality in \eqref{eq:nr+R<=n+1D/2}, where the extreme cases are the triangle $S$ and its reverse $-S$. 
Finally, the family of triangles $T$ from Lemma \ref{lem:Isosceles_Extreme_Triangle}, for some $D\in[1,2]$,
attains equality in \eqref{eq:triangles_and_triangles} and their extreme cases are a line segment $[x,y]$ (when $D=2$) and the triangle $-S$ (when $D=1$).

This means that the points described by the equality cases of those three inequalities form a simple closed curve which contains $f(\CK^2,S)$ and for each boundary point, there is a convex body $K$ mapped onto it via $f(K,S)$. 
Moreover, let $T$ be such that $f(T,S)$ belongs to the boundary induced by Lemma \ref{lem:New_ineq_isos_triangles_Triang}. 
Due to the homogeneity of the first entry of $r$ and $D$ we have that
\[
f(T,S)=\left(\frac{r(T,S)}{R(T,S)},\frac{D(T,S)}{2R(T,S)}\right) =
\left(r\left(\frac{T}{R(T,S)},S\right),\frac12D\left(\frac{T}{R(T,S)}\right)\right)=f\left(\frac{T}{R(T,S)},S\right).
\]
Hence we can replace $T$ by $T/R(T,S)$ and therefore assume that $R(T,S)=1$.
Finally, since $f(T,S)$ is linear in the first entry with respect to the Minkowski addition of $T$ and $S$ (see Lemma \ref{lem:starshaped})
and \eqref{eq:D<=2R} and the inequality in Theorem \ref{cor:nr+R<=n+1D/2} are linear with endpoint at $S$, thus the region contained between these three boundaries \emph{coincides} with the diagram $f(\mathcal K^2,S)$, since it is completely covered by the values
\[
f((1-\lambda)T+\lambda S,S)=(1-\lambda)f(T,S)+\lambda(1,1)^T,\quad \lambda\in[0,1],
\]
for every $D\in[1,2]$, which concludes the proof of Theorem \ref{thm:dominating_diagram}.

Theorem \ref{thm:BS_diag_simplex} now is a straightforward consequence of the fact that all the inequalities in Theorem \ref{thm:dominating_diagram} are already best possible if we choose $S$ to be a triangle.
\end{proof}

%\begin{proof}[Proof of Theorem \ref{thm:dominating_diagram}]

%the inequalities \eqref{eq:D<=2R}, \eqref{eq:nr+R<=n+1D/2}, and \eqref{eq:New_Ineq_isosceles_triangle}, together with Theorem \ref{thm:BS_diag_simplex}.
%\end{proof}

%\section{Special diagrams induced by (regular) $k$-gons with $k=4,5,6$}
\section{Importance of $k$-gons in the planar case}\label{sec:456}

%\RSays{**I like the intro with the union, but would suggest to use it also to motivate the square describing the intersection $\bigcap_{C\in\CK^2} f(\CK^2,C)$. The problem here is more that we motivate a wrong claim that we already told to be wrong in the introduction. Maybe we can switch a bit with the introduction or just write here more directly the truth**}
%\BSays{**I like the idea of the intersection! Regarding the wrong claim and so on, yes, we already did in the intro, here we should be shorter - of course, it was not really motivated in the original version, this is why it was motivated here**}

From the previous sections we learned that
\begin{equation}\label{eq:union_diagrams_general_case}
\bigcup_{C\in\CK^2} f(\CK^2,C) = f(\CK^2,\CK^2)=f(\CK^2,S) \quad \text{and} \bigcap_{C \in \CK^2} f(\CK^2,C) = f(\CK^2,P)
 \quad  \end{equation}
for some triangle $S$ and some parallelogram $P$.
In view of the first %Theorem \ref{thm:dominating_diagram} 
in combination with Lemma \ref{lem:opt_two_simplices} one may naturally conjecture that
%\begin{equation}\label{eq:union_symm_containers}
$\bigcup_{C\in\mathcal K_0^2} f(\mathcal K^2,C) = f(\mathcal K^2,C_0)$,
%\end{equation}
for some $C_0\in\CK^2_0$. However, %despite \eqref{eq:union_diagrams_general_case}, 
the symmetric case is a bit different as Theorem \ref{thm:hexagons} shows.

%\begin{lem}\label{lem:large_inradius_hexagon}
%Let $K,C\in\mathcal K^2$ be such that $C$ is a regular hexagon. If $D(K,C)<2R(K,C)$, then $r(K,C)>R(K,C)/4$. 
%\end{lem}

\begin{proof}[Proof of Theorem \ref{thm:hexagons}]
After a suitable translation and dilatation of $K$, let $K\subset^{opt} H$ Now, by Proposition \ref{prop:opt_cont_criterion} there exist $p^i\in K\cap \bd(H)$ and $u^i\in N(H,p^i)$ such that $0 \in \conv(\set{u^1,u^2,u^3})$ and 
%due to the facial structure of $H$, since $0=\sum_{i=1}^3\lambda_iu_i$ and
since $D(K,H)<2R(K,H)$ we must have that the points $p^i$ belong to  non-consecutive edge of $H$.

We now show that $r(K,H)\geq R(K,H)/4$. To do so, let $q^1,\dots,q^6$ be the vertices of $H$ in clockwise order.
Moreover, let us suppose that each $p^1$ belongs to $[q^2,q^3]$, $p^2$ to $[q^4,q^5]$, and $p^3$ to $[q^6,q^1]$.
Notice that the length of 
\[
\conv(\{q^1,q^2,\frac{1}{2}(q^4+q^5)\}) \cap [q^3,q^6]
\]
is half of the length of $[q^1,q^2]$. The same holds when considering the triangle of vertices
$q^3,q^4,\frac12(q^6+q^1)$ and segment $[q^2,q^5]$ or the triangle of vertices $q^5,q^6,\frac12(q^2+q^3)$ and segment
$[q^1,q^4]$. Therefore, those three triangles optimally contain a factor $\frac14$ dilatation of $C$.

Now, assume that $c_1+r_1H\subset^{opt} K$, for some $c_1\in H$ and $r_1< \frac14$. In particular,
this means that $c+rH \subset^{opt}\conv(\set{p^1,p^2,p^3})$, for some  $r< 1/4$ and $c\in H$. 
%for which we may assume $(u^1)^Tc \le 0$ \BSays{, i.e. $(q^5+q^6)^Tc\geq 0$.}
Moreover, since $D(K,H)<2R(K,H)$ we may assume without loss of generality, $p^1 \in (q^2,q^3), p^2 \in (q^4,q^5)$, and $p^3 \in (q^1,q^6)$ and because of the symmetry group of $H$ that $c=tq^4 + su^1$, for some $t,s \geq 0$.

\begin{figure}
    \centering
    \includegraphics[width=8cm]{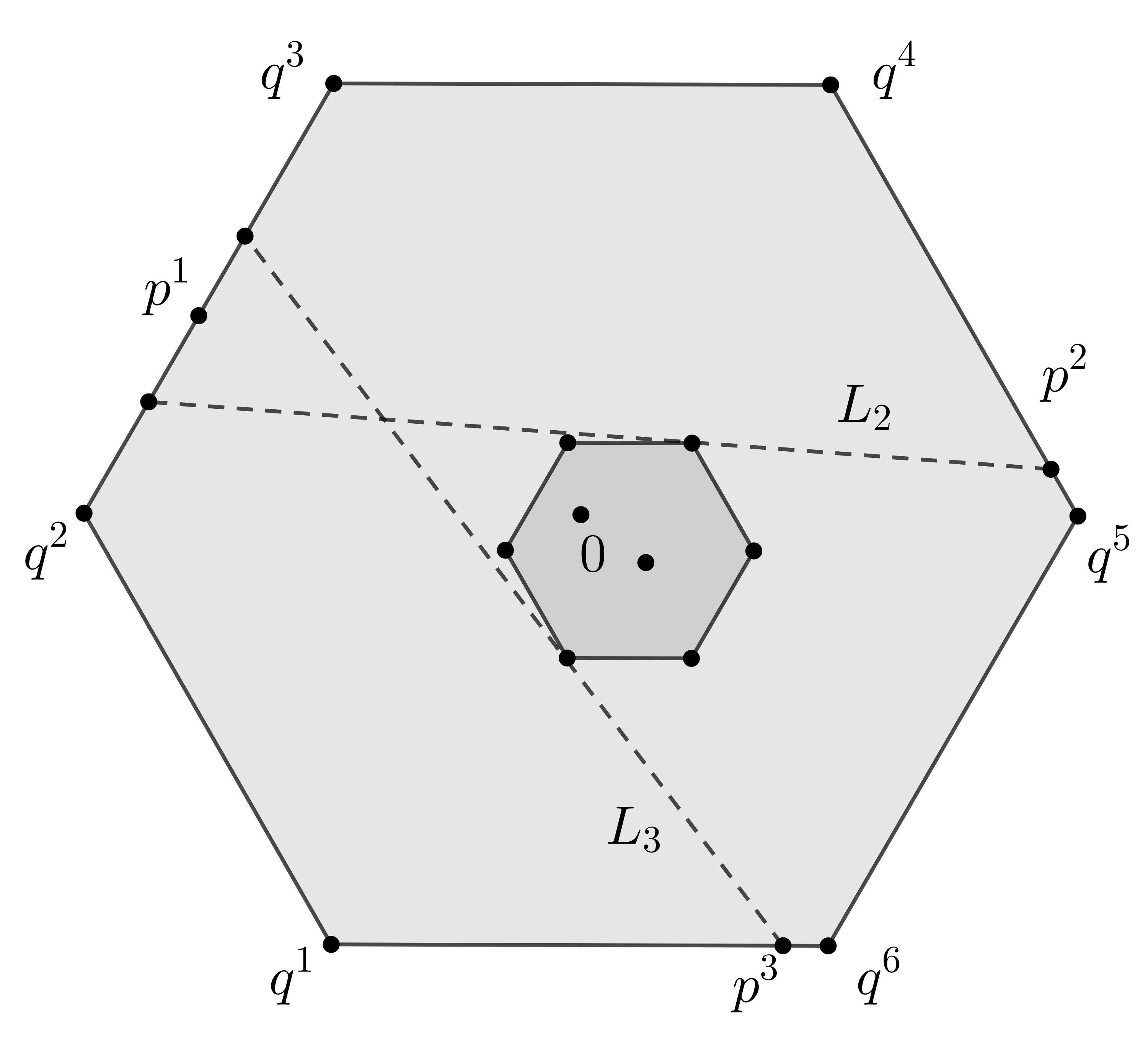}
    \caption{Trying to build a triangle (with vertices $p^1,p^2,p^3$) optimally contained within a hexagon $H$, such that it optimally contains itself a homothet of $H$ rescaled by a factor less than $1/4$.
    %\BSays{Homothetic regular hexagons $H$ and $c+rH$, $r<1/4$, and supporting lines $L_2$, $L_3$ that cannot intersect at a common boundary point $p^1$ of $H$.}
    }
    \label{fig:HexagonProof}
\end{figure}

Now, let $L_2$ and $L_3$ be line segments supporting $c+rH$ with one endpoint at $p^2$ and $p^3$, respectively, such that $c+rH$ and $q^6$ are on the same halfspace induced by $L_2$, while $c+rH$ and $q^5$ are on the same halfspace induced by $L_3$ (\cf~Figure \ref{fig:HexagonProof}).
Notice that $c+rH$ cannot be contained in the halfspace determined by $[q^3,q^6]$ 
%where $q^4$ and $q^5$ stay - 
as otherwise, $L_2$ would not hit the boundary of $H$ again within $(q^2,q^3)$, contradicting that $c+rH$ is the inball of $\conv(\set{p^1,p^2,p^3})$.
%either $[p^1,p^3]$ would not touch $c+rH$, or if touching $c+rH$, would be forced to fulfill $p^1=q^3$ and $p^3=q^6$, contradicting the fact that $D(K,H)<2R(K,H)$. 
The same reason implies that $c+rH$ cannot be contained in the halfspace determined by $[q^2,q^5]$. %where $q^1$ and $q^6$ stay.

In the remainder of the proof we want to show that $L_2$ and $L_3$ cannot both go through the same point $p^1 \in (q^2,q^3)$, a contradiction again.

Let $L_2'$ and $L_3'$ be line segments supporting $rH$ with endpoints at $q^6$ and $q^5$, respectively, such that $q^5$ and $rH$ belong to the same halfspace induced by $L_2'$, while $q^6$ and $rH$ belong to the same halfspace induced by $L_3'$. Since $r<1/4$ this means that $L_2'$ intersects the boundary of $H$ at a point closer to $q^4$ (in terms of walking along the boundary of $H$) than $L_3'$. 
This last property will stay true on each of the next geometric steps of the proof.

Now, notice that replacing $rH$ by $tq^4+rH$ and changing $L_2'$ and $L_3'$ accordingly, the two lines support $tq^4+rH$ at $tq^4+rq^1$ and $tq^4+rq^4$, respectively. Since $[q^4,q^1]$ is parallel to $[q^5,q^6]$, we still have that $L_2'$ intersects the boundary of $H$ at a point closer to $q^4$ than $L_3'$.
Again, replacing $tq^4+rH$ by $tq^4+su^1+rH = c+rH$, and changing $L_2'$ and $L_3'$ correspondingly, moves the intersection point of $L_2'$ with the boundary of $H$ even closer to $q^4$ and the one of $L_3'$ even further away from $q^4$.
Finally, replacing $L_2'$ by $L_2$ and $L_3'$ by $L_3$ again moves the intersection point with the boundary of $H$ of the former closer to $q^4$ and the intersection point of the latter further away from $q^4$. 
Since at the beginning $L_2'$ and $L_3'$ could not intersect at a common boundary point of $H$, we can conclude they can neither at the end. Hence $L_2$ and $L_3$ cannot intersect at any point of the boundary of $H$, arriving at the desired contradiction.
Thus we must have $r(K,H) \geq R(K,H)/4$.
\end{proof}

\begin{rem}
Notice that if $H$ is a regular hexagon, then \eqref{eq:D<=2R} and \eqref{eq:r+R<=D} induce boundaries of the diagram $f(\CK^2,H)$. Moreover, assuming that $H=\conv(\set{q^1,\dots,q^6})$,
%$H=\mathrm{conv}(\{(\cos(2\pi m/6),\sin(2\pi m/6)):m\in[6]\})$ 
we conjecture that the family of isosceles triangles
\[
T_\lambda:=\conv(\set{(1-\lambda)q^1+\lambda q^2,(1-\lambda)q^4+\lambda q^3,\frac12(q^5+q^6)})
%\mathrm{conv}(\{(1-\lambda)(1,0)^T+\lambda(\frac12,\frac{\sqrt{3}}{2})^T, (1-\lambda)(-1,0)^T+\lambda(-\frac12,\frac{\sqrt{3}}{2})^T,(0,-\frac{\sqrt{3}}{2})^T\})
\]
induces the left boundary of the diagram $f(\CK^2,H)$. Doing a detailed computation one can check that
$f(T_\lambda,H)=((\lambda+1)(2-\lambda)/(4+\lambda),(1+\lambda)/2)$, or explicitly, that $r(T_\lambda,H)=D(T_\lambda,H)(3-D(T_\lambda,H))/(3+D(T_\lambda,H))$ (\cf~Figure \ref{fig:hexagon}).
%and \eqref{eq:Bohnenblust} and Theorem \ref{thm:hexagons} provide bounds for the diagram $f(\mathcal K^2,H)$ (see Figure \ref{fig:hexagon}).
\end{rem}

\begin{figure}
    \centering
    \includegraphics[width=8cm]{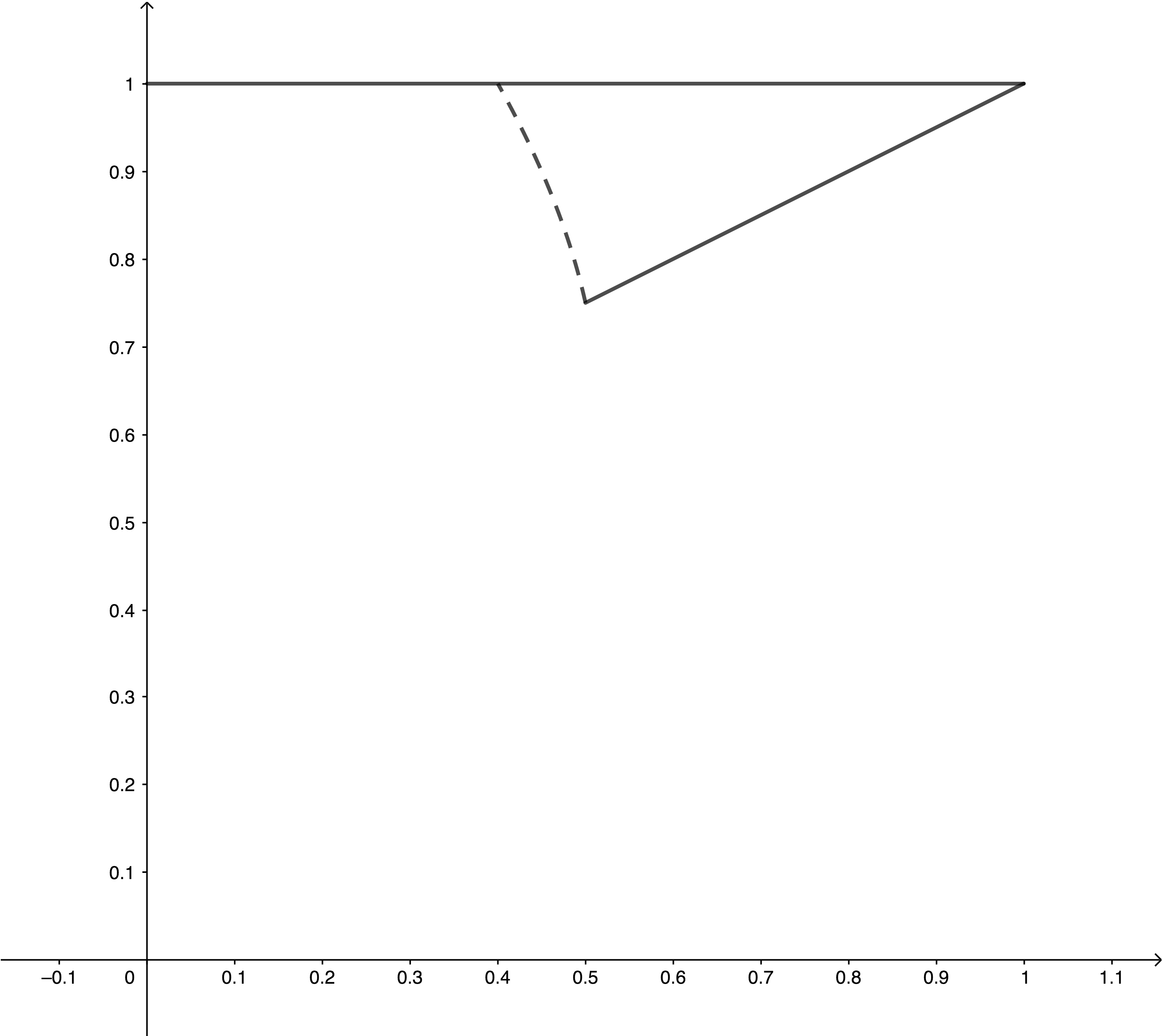}
    \caption{The diagram $f(\mathcal K^2,H)$, $H$ regular hexagon.}
    \label{fig:hexagon}
\end{figure}

On the one hand, Theorem \ref{thm:hexagons}, together with the fact that there exist equilateral triangles $T$ such that $f(T,\B_2) = (x,y)$ with $y<1$ and $x$ arbitrary close to 0, shows that the union $f(\CK^2,\CK^2_0) = \bigcup_{C \in \CK^2_0} f(\CK^2,C)$ \emph{cannot} be covered by $f(\CK^2,H)$, $H$ a regular hexagon. On the other hand, the equality case of \eqref{eq:Bohnenblust} shows that $(1/2,3/4)^T \in f(\CK^2,C)$, $C\in\CK^2_0$, if and only if $C$ is an affine transformation of a regular hexagon. Hence $f(\CK^2,\CK^2_0)$ cannot be given by a single symmetric container.

%\begin{cor}
%The union of diagrams $f(\mathcal K^2,\mathcal K^2_0)$ does not coincide %with the diagram provided by any symmetric container.
%\end{cor}

%\begin{proof}
%Let $S\in\mathcal K^2$ be a triangle, and let $T$ be an (Euclidean) isosceles triangle, with two edges longer than the 
%third one, and such that $R(T,\mathbb B_2)=1$ and $r(T,\mathbb B_2)=r<1/4$.

%On the one hand, notice that by \eqref{eq:Bohnenblust} (see also \cite{BrGM17_3}) 
%we know that $(x(S,S-S),y(S,S-S))=(1/2,3/4)$. Moreover, 
%by \cite[Cor. 2.9]{BrGM17_3} $y(K,C)=3/4$ for some $C\in\mathcal K^2_0$ 
%if and only if $K$ is a triangle and $C$ is a dilatation of $K-K$ (see \eqref{eq:equalityBohnenblust}
%and the fact that in the planar case $K-K=3(K\cap(-K))$). 
%This means that the union $\cup_{C\in\mathcal K^2_0} f(\mathcal K^2,C)$, in order to cover $f(S,S-S)$, cannot omit $f(\mathcal K^2,S-S)$.

%On the other hand, by Lemma \ref{lem:opt_two_simplices}, there exists a %triangle $T_1$ such that
%$T\subset^{opt}\mathbb B_2\subset^{opt}T_1\cap (-T_1)$ with $T\subset^{opt}T_1\cap(-T_1)$. In particular,
%$r(T,T_1\cap(-T_1)) \leq r(T,\mathbb B_2)=r<1/4$. Moreover, $D(T,\mathbb B_2)/2<1$. Therefore,
%Theorem \ref{thm:hexagons} implies that
%$f(T,\mathbb B_2)$ does not belong to $f(\mathcal K^2,S-S)$. 

%Thus, we have shown that 
%\[
%f(S,S-S)\in f(\mathcal K^2,S-S)\setminus \bigcup_{C\neq S-S}f(\mathcal K^2,C)\quad\text{ but }\quad f(T,\mathbb B_2)\notin f(\mathcal K^2,S-S),
%\]
%concluding the proof.
%\end{proof}

\begin{rem}\label{rem:union_hexagons}
With a bit more effort, we can also prove, by means of a new geometric property (that we will show in a forthcoming paper)
together with Lemma \ref{lem:opt_two_simplices}, the non-obvious fact that
\[f(\CK^2,\CK^2_0)=\bigcup_{S\text{ a triangle}}f(\CK^2,S\cap(-S)).\]
\end{rem}

We now turn to the case in which the diagram is considered with respect to a regular pentagon.

\begin{proof}[Proof of Theorem \ref{thm:pentagons}]
Let $p^1,\dots,p^5$ be the vertices of a regular pentagon $P$ with $\|p^i\|=1$, $i=1,\dots,5$, and $p^{ij}=(p^i+p^j)/2$, $1\leq i<j\leq 5$.
We start noticing that $-P\subset s(P)P$, where we clearly have that 
$s(P)=\|p^1\|/\|p^{34}\|=\sqrt{5}-1$. Notice also that since the diameter of a triangle is attained along an edge, 
selecting $T := \conv(\set{p^{34},p^{15},p^{12}})$ we obtain
\[
D([p^{15},p^{12}],P)<D([p^{34},p^{15}],P)=D(-\frac{P}{s(P)},P)=\frac{2}{\sqrt{5}-1}=\frac{\sqrt{5}+1}{2},
\]
and so $D(T,P)=(\sqrt{5}+1)/2$.
Moreover, an analogous geometric argument shows that for any $x \in [p^5,p^{15}]$ we have $D([x,p^{34}],P)=(\sqrt{5}+1)/2$ too.

Now, let $K\in\CK^2$. After a suitable dilatation and translation, we may assume that $K \subset^{opt} P$. Using Proposition \ref{prop:opt_cont_criterion}, there exist $q^1,q^2,q^3 \in K \cap \bd(P)$ such that, after reordering if needed, $q^1 \in [p^3,p^4]$, $q^2 \in [p^1,p^5]$, and $q^3 \in [p^1,p^2]$. 

Without loss of generality we may assume that $q^1$ is closer to $p^3$ than to $p^4$. Moreover, let $x\in[p^1,p^5]$. Then
\[
\begin{split}
    D([x,q^1],P) & \geq 2\frac{h([x-q^1,q^1-x],p^3)}{h(P-P,p^3)} \\
    & \geq 2\frac{h([x-p^{34},p^{34}-x],p^3)}{h(P-P,p^3)} \\
    & \geq 2\frac{h([p^{15}-p^{34},p^{34}-p^{15}],p^3)}{h(P-P,p^3)} \ge \frac{\sqrt{5}+1}{2},
\end{split}
\]
This shows that $D(K,P) \geq D([q^2,q^1],P) \geq (\sqrt{5}+1)/2$, i.e.
\[
\frac{D(K,P)}{R(K,P)} \geq \frac{\sqrt{5}+1}{2} = \frac{D(T,P)}{R(T,P)},
\]
concluding that $j_P=(\sqrt{5}+1)/2$.

In order to finish the proof, let 
\[
T_\lambda:=\conv(\set{(1-\lambda)p^5+\lambda p^1,(1-\lambda)p^2+\lambda p^1,p^{34}}),
\]
$\lambda\in[0,1/2]$ (which means $T=T_{\frac 12})$. From the computations above, we know that 
\[
D([(1-\lambda)p^5+\lambda p^1,p^{34}],P)=D([(1-\lambda)p^2+\lambda p^1,p^{34}],P)=\frac{\sqrt{5}+1}{2},
\]
for every $\lambda\in[0,1/2]$. Moreover,
since 
\[
D([(1-\lambda)p^5+\lambda p^1,(1-\lambda)p^2+\lambda p^1],P) 
\left\{\begin{array}{lr}
    <\frac{\sqrt{5}+1}{2} & \text{if } \lambda=1/2,\\
    =2 & \text{if }\lambda=0, 
\end{array}\right.  
\]
by continuity there exists $\lambda_0\in(0,1/2)$ such that $D([(1-\lambda_0)p^5+\lambda_0 p^1,(1-\lambda_0)p^2+\lambda_0 p^1],P) 
=(\sqrt{5}+1)/2$. 
Such $\lambda_0$ fulfills in particular that $\|p_{34}-p^{15}\|=\|(1-\lambda_0)p^5+\lambda_0 p^1-(1-\lambda_0)p^2-\lambda_0p^1\|$, i.e. 
$(1-\lambda_0)\|p^5-p^2\|=\|p^{34}-p^{15}\|$ and thus
\[
1-\lambda_0=\frac{\|p^{34}-p^{15}\|}{\|p^5-p^2\|}=\frac{\|p^{34}-p^{15}\|}{\|p^3-p^1\|}=\frac{D(T_1,P)}{2}=\frac{\sqrt{5}+1}{4},
\]
from which $\lambda_0=(3-\sqrt{5})/4$. For such $\lambda_0$, define $T':=T_{\lambda_0}$. Then all three edges of $T'$ are diametrical, which means $T'$ is equilateral and $R(T',P)/D(T',P) = j_P$, which concludes the proof.
\end{proof}

\begin{rem}
Notice that using the notation of Theorem \ref{thm:pentagons}, 
%and letting $P$ be a regular pentagon, 
$s(P)=\sqrt{5}-1$ together with \eqref{eq:sr+R<=s+1D/2} gives
\begin{equation} \label{eq:right-pentagon}
(\sqrt{5}-1)r(K,P)+R(K,P)\leq\frac{\sqrt{5}}{2}D(K,P)
\end{equation}
for every $K\in\CK^2$ and equality holds in the above inequality for $(1-\lambda)(-P)+\lambda P$, $\lambda\in[0,1]$ and especially for $K=-P$ we have equality in \eqref{eq:sr+R<=s+1D/2} for a set, which also achieves the Jung-constant $j_P$. 
Moreover, we conjecture that the family filling the left boundary of the diagram $f(\CK^2,P)$
is given by the isosceles triangles
\[
%\begin{split}
T_\lambda 
%& 
:=
\conv(\{(1-\lambda)p^1+\lambda p^2,(1-\lambda)p^1+\lambda p^5,p^{34}\}) 
%\\
%& = \conv(\{(1-\lambda)(1,0)^T+\lambda(\cos\frac{2\pi}{5},\pm\sin\frac{2\pi}{5})^T,(\cos\frac{4\pi}{5},0)^T\})
,
%\end{split}
\]
%\RSays{**Again I suggest only to describe the geometry and remove details like coordinates -- here this means: we keep the first but remove the second line of the definition above.**}\BSays{**Okay. And below?**}\RSays{**we can still write "after a detailed computation"**}
for $\lambda\in[0,1/2]$. After a lengthy computation, their image equals
$f(T_\lambda,P)=(\lambda(1+\lambda\frac{\sqrt{5}-3}{2}),1+\lambda\frac{\sqrt{5}-3}{2})^T$.
Explicitly, $r(T_\lambda,P)=\frac{2}{\sqrt{5}-3}(D(T_\lambda,P)/2-1)D(T_\lambda,P)/2$.
Gathering this information, the Jung-constant of $P$ given in Theorem \ref{thm:pentagons}, \eqref{eq:D<=2R}, and \eqref{eq:right-pentagon}, we obtain the (in one boundary conjectured) diagram $f(\CK^2,P)$ (\cf~Figure \ref{fig:pentagon}).
%Finally, a more lengthy computation shows that $r(T_1,P)=2/5$. Gathering this information, together with
%its Jung's inequality (see Theorem \ref{thm:pentagons}), explains some boundaries of the 
%diagram $f(\mathcal K^2,P)$ (see Figure \ref{fig:pentagon}).
\end{rem}

\begin{figure}
    \centering
    \includegraphics[width=8cm]{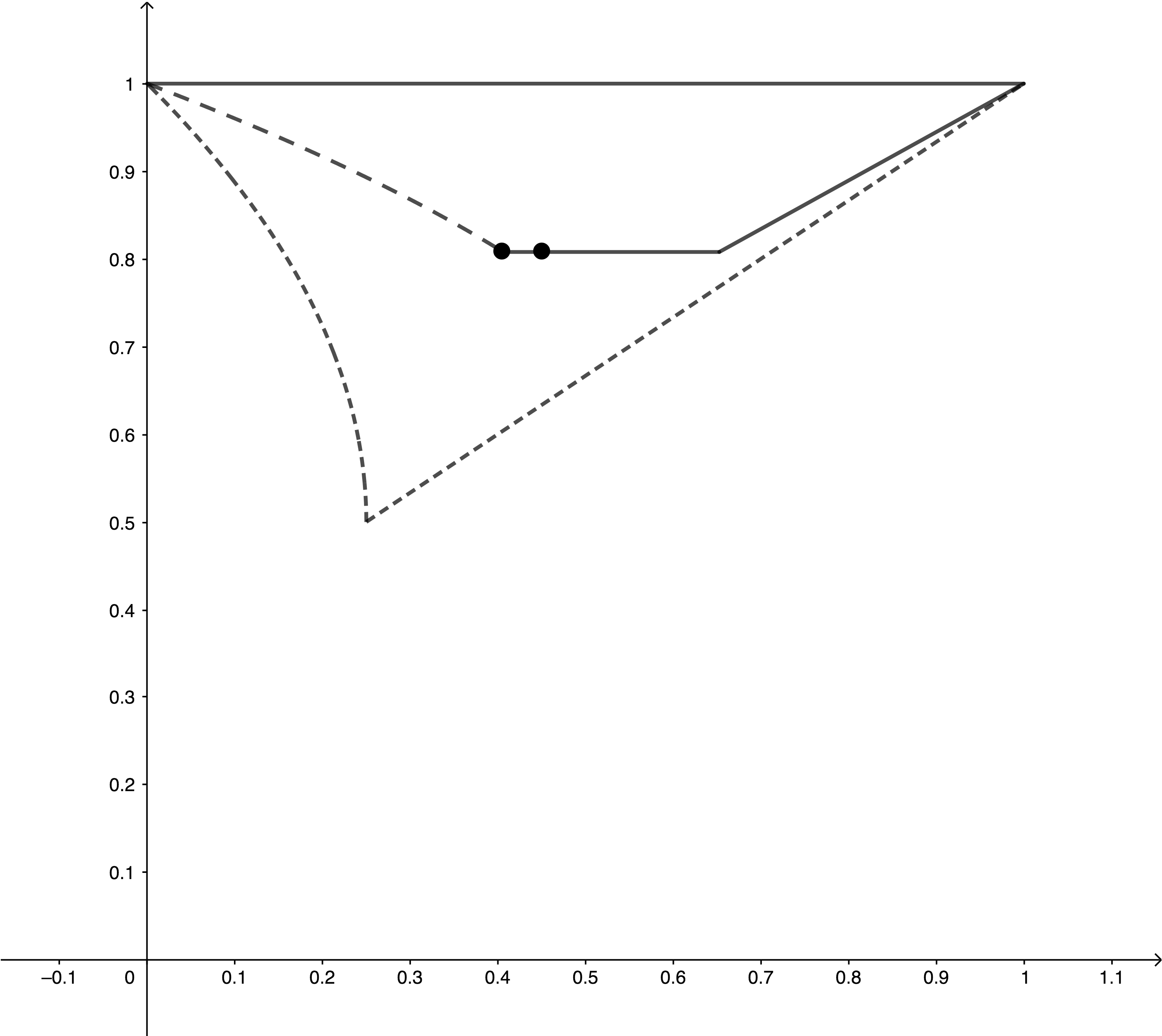}
    \caption{The diagram $f(\CK^2,P)$, with $P$ a regular pentagon, and the set $f(\CK^2,\CK^2)$. The two bold points indicate the coordinates of the isosceles triangle $T$ (left one) and the equilateral $T'$ (right).}
    \label{fig:pentagon}
\end{figure}

%\RSays{**Is here anything to remark about non-regular pentagons? Some weird ones? (Keeping affinity class representatives, s.t.~the regular triangle is Jung-extreme**}

\begin{rem}
Let us denote by $\CC^k$ and  $\CC^k_0$ the families of all $k$-gons and all symmetric $k$-gons, respectively. Then we can find sequences $(C_m)_{m \in \N}$ in $\CC^k$ or in $\CC^k_0$ such that $C_m \rightarrow S$ or $C_m \rightarrow S\cap(-S)$, ($m\rightarrow \infty$), respectively, for some triangle $S$ and with respect to the Hausdorff metric. Doing so the continuity of the radii functionals together with Theorems \ref{thm:BS_diag_simplex} and \ref{thm:dominating_diagram} (see also Remark \ref{rem:union_hexagons}) would immediately imply that
\[
f(\CK^2,\CK^2)=\overline{\bigcup_{C\in\mathcal C^k}f(\CK^2,C)} 
\quad and \quad f(\CK^2,\CK_0^2)=\overline{\bigcup_{C\in\mathcal C^k_0}f(\CK^2,C)},
\]
for every $k\geq 3$ or $k\geq 6$, respectively.
Moreover, in case $k \ge 4$ we could choose the sequence $(C_m)_{m \in \N}$ such that it converges against a parallelogram thus showing together with Theorem \ref{thm:squares} that 
\[\overline{\bigcap_{C\in\CC^k}f(\CK^2,C)} = \conv\left(\set{\left(\begin{smallmatrix} 0 \\ 1 \end{smallmatrix}\right),\left(\begin{smallmatrix} 1 \\ 1 \end{smallmatrix}\right)}\right)\]
\end{rem}

%\RSays{**It seems that we do not cite [7,16,21]**}

%%%%%%%%%%%%%%%%%%%%%%%%%%
%%%%%%%%%%%%%%%%%%%%%%%%%%

\end{document}